%% file: RSK-isomorphism.tex
\documentclass[12pt]{article}

\usepackage[utf8]{inputenc}
\usepackage[T1]{fontenc}

\usepackage{subfiles}

\usepackage{amssymb,amsmath,amsthm,commath,xspace,mathdots}
\usepackage{enumitem}
	
\usepackage{aliascnt}  

\usepackage[colorlinks,citecolor=blue,urlcolor=blue]{hyperref}

\usepackage[capitalize,noabbrev]{cleveref}

\usepackage{times,bbm}

\usepackage[backgroundcolor=yellow]{todonotes}

\usepackage[labelfont={normalsize}]{caption,subfig}

\usepackage{tikz}
\usepackage{tikz-cd}
\usetikzlibrary{calc,colorbrewer,decorations.markings}
\usetikzlibrary{backgrounds,patterns}
\usetikzlibrary{patterns.meta}
\usetikzlibrary{shapes.arrows}

\usepackage{pgfplots}
\pgfplotsset{compat=1.18}

\makeatletter
\@ifclasswith{article}{externalize}%
{
	\usetikzlibrary{external} 
	\tikzexternalize[
	optimize=false,
	prefix=tikz/,
	mode=list and make]
	\makeatletter
	\renewcommand{\todo}[2][]{\tikzexternaldisable\@todo[#1]{#2}\tikzexternalenable}
	\makeatother
}%
{
	\newcommand{\tikzexternaldisable}{}
	\newcommand{\tikzexternalenable}{}
}
\makeatother

\usepackage{xcolor} 

\newcommand{\R}{\mathbb{R}}
\newcommand{\N}{\mathbb{N}}

\DeclareMathOperator{\RSK}{RSK}
\newcommand{\RSKcos}{\operatorname{RSKcos}}
\newcommand{\RSKsin}{\operatorname{RSKsin}}
\newcommand{\scq}{q}

\newcommand{\tree}{\tau}
\newcommand{\origin}{\mathcal{O}}

\newcommand{\jdtpath}{\mathbf{p}}
\newcommand{\jdt}{J}
\newcommand{\jdtforest}{\Phi}

\newcommand{\asymdirect}{\alpha_\textnormal{asym}}

\newcommand{\cusp}{\operatorname{cusp}}

\newcommand{\gridcommand}{}

\theoremstyle{definition}
\newtheorem{definition}{Definition}[section]

\theoremstyle{plain}
\newaliascnt{theorem}{definition}
\newtheorem{theorem}[theorem]{Theorem}
\aliascntresetthe{theorem}

\newaliascnt{proposition}{definition}
\newtheorem{proposition}[proposition]{Proposition}
\aliascntresetthe{proposition}

\newaliascnt{corollary}{definition}
\newtheorem{corollary}[corollary]{Corollary}
\aliascntresetthe{corollary}

\newaliascnt{lemma}{definition}
\newtheorem{lemma}[lemma]{Lemma}
\aliascntresetthe{lemma}

\newaliascnt{assumption}{definition}
\newtheorem{assumption}[assumption]{Assumption}
\aliascntresetthe{assumption}

\newaliascnt{conjecture}{definition}

\aliascntresetthe{conjecture}

\newaliascnt{problem}{definition}

\aliascntresetthe{problem}

\theoremstyle{remark}
\newaliascnt{example}{definition}

\aliascntresetthe{example}

\newaliascnt{remark}{definition}
\newtheorem{remark}[remark]{Remark}
\aliascntresetthe{remark}

\newaliascnt{question}{definition}

\aliascntresetthe{question}

\crefname{theorem}{Theorem}{Theorems}
\Crefname{theorem}{Theorem}{Theorems}
\crefname{proposition}{Proposition}{Propositions}
\Crefname{proposition}{Proposition}{Propositions}
\crefname{corollary}{Corollary}{Corollaries}
\Crefname{corollary}{Corollary}{Corollaries}
\crefname{lemma}{Lemma}{Lemmas}
\Crefname{lemma}{Lemma}{Lemmas}
\crefname{assumption}{Assumption}{Assumptions}
\Crefname{assumption}{Assumption}{Assumptions}
\crefname{conjecture}{Conjecture}{Conjectures}
\Crefname{conjecture}{Conjecture}{Conjectures}
\crefname{example}{Example}{Examples}
\Crefname{example}{Example}{Examples}
\crefname{remark}{Remark}{Remarks}
\Crefname{remark}{Remark}{Remarks}
\crefname{question}{Question}{Questions}
\Crefname{question}{Question}{Questions}
\crefname{problem}{Problem}{Problems}
\Crefname{problem}{Problem}{Problems}

\usepackage[backend=bibtex, doi=true, isbn=false, url=false, style=alphabetic,
giveninits=true, maxalphanames=4, 
maxnames=5]{biblatex}
\AtEveryBibitem{\clearlist{language}}

\newcommand{\biblio}{\printbibliography}

\title{Jeu de taquin forests \\ and the inverse infinite RSK correspondence}

\author{Dan Romik and Piotr Śniady}
\date{}

\begin{document}

\maketitle

\begin{abstract}
	The \emph{Plancherel-random infinite Young tableau} arises
	from applying the infinite Robinson--Schensted--Knuth (RSK)
	correspondence to a sequence of i.i.d.~random variables
	distributed uniformly on the unit interval $[0,1]$. Building on the
	isomorphism of dynamical systems established in prior work, we
	provide an explicit geometric characterization of the
	inverse map. Our approach makes use of a previously unexplored structure, 
	the \emph{jeu de taquin forest} on the infinite tableau, where edges connect
	boxes according to local comparison rules. We prove that
	each tree in this forest almost surely extends toward
	infinity with a well-defined asymptotic direction, establishing a
	canonical bijection between trees and values in the i.i.d.~input 
	sequence. The ordering is recovered through tree lifetimes under iterated
	jeu de taquin transformations. 
\end{abstract}

\section{Introduction}

\label{sec:intro}

\subsection{The RSK and jeu de taquin maps on infinite tableaux}

We recall the main definitions and results from \cite{RomikSniady2015}.
An \emph{Infinite Young Tableau} (IYT) is an array
\[ T = (T_{x,y})_{x,y\ge0} \]
where the entries $T_{x,y}$ are increasing along rows and
columns and take positive integer values, such that $T$
contains each positive integer exactly once. We visualize
$T$ graphically as an infinite array of \emph{boxes}, where
each pair $(x, y) \in \N_0^2$ of indices is referred to as a
``box'' and is associated with (and treated as synonymous
with) the square $[x,x+1]\times [y,y+1]$ that is depicted as
containing the entry~$T_{x,y}$. See \cref{fig:iyt} for an
example.

\begin{figure}[t]
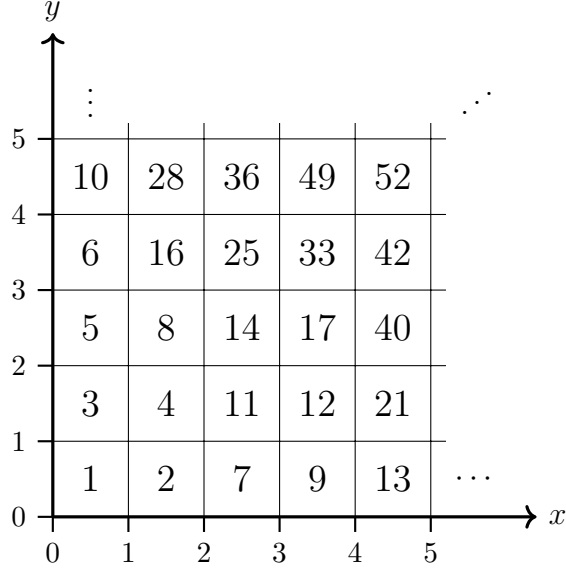

	\centering
	\subfile{figures/FIGURE-iyt.tex}
	\caption{An infinite Young tableau.}
	\label{fig:iyt}
\end{figure}

We denote by $\Omega$ the
space of infinite Young tableaux. This space (viewed as a
measurable space in the usual way by equipping it with the
minimal $\sigma$-algebra $\mathcal{F}$ that turns all
coordinate functions $T\mapsto T_{x,y}$ into measurable
functions) is equipped with a natural probability measure
$P$, known as \emph{(infinite) Plancherel measure}, that is
characterized by the property that for each (finite) Young
diagram $\lambda$ and standard Young tableau $Q =
(q_{x,y})_{(x,y) \in \lambda}$ of shape $\lambda$, we have
\[
P\Big( T_{x,y} = q_{x,y} \textrm{ for all }(x,y) \in \lambda \Big) =
\frac{f^\lambda}{n!},
\]
where $n=|\lambda|$ (the number of boxes in $\lambda$) and
$f^\lambda$ denotes the number of standard Young tableaux of
shape $\lambda$.

The Plancherel measure space $(\Omega, \mathcal{F}, P)$ has
a deep connection to two fundamental operations from
algebraic combinatorics and representation theory: the
\emph{Robinson--Schensted--Knuth (RSK) map}
\cite{Robinson1938,Schensted1961,Knuth1970}, and the \emph{jeu
	de taquin (JDT) map} of Schützenberger
\cite{Schuetzenberger1977}. An infinite version of the RSK
correspondence was first considered, in a more general setup,
by Kerov and Vershik \cite{KerovVershik1986} in their study
of the characters of the infinite symmetric group. In the
infinite tableau setting, the two operations are defined as
follows:

\begin{itemize}

\item The RSK map is a map $\RSK:[0,1]^\N \to \Omega$
defined as the projective limit of the RSK maps on finite
sequences of real numbers, in the following sense: let
$\mathbf{z}=(z_n)_{n=1}^\infty \in [0,1]^\N$. For each
$n\ge1$, let $(\lambda^{(n)},P_n,Q_n)$ be the triple
associated to the finite sequence $z_1,\ldots,z_n$ by the
usual RSK correspondence; this means that $\lambda^{(n)}$ is
a Young diagram with $n$ boxes (the ``RSK shape''), $P_n$ is
a semistandard Young tableau of shape $\lambda^{(n)}$ with the
entries $z_1,\ldots,z_n$ (the ``insertion tableau''), and
$Q_n$ is a standard Young tableau of shape $\lambda$ (the
``recording tableau''). Then
$\RSK(\mathbf{z})$ is defined as the
(necessarily unique) IYT $T=(T_{x,y})_{x,y\ge0}$ whose
restriction to the boxes of the Young diagram
$\lambda^{(n)}$ is equal to $Q_n$ for any $n\ge1$, if such
an IYT exists.

The value $\RSK(\mathbf{z})$ is well-defined for
\emph{almost every} (but not every) input $\mathbf{z}$, with
respect to the product of Lebesgue measures on $[0,1]^\N$.

\item The jeu de taquin map is a map $J:\Omega\to\Omega$
that acts on an IYT through the following sequence of
operations (see \cref{fig:jdt-forest}):
\begin{enumerate}

\item remove the box with entry $T_{0,0}=1$ at position $(0, 0)$. 

\item Perform Schützenberger promotion, consisting of a
sequence of sliding moves involving the chain of boxes with
positions
\begin{equation} \label{eq:jdt-path}
(x_0, y_0) \leftarrow (x_1,y_1) \leftarrow (x_2, y_2) \leftarrow \ldots
\end{equation}
defined by the recurrence
\[
(x_n,y_n) = \begin{cases}
(x_{n-1}+1,y_{n-1}) & \textrm{if }T_{x_{n-1}+1,y_{n-1}} < T_{x_{n-1},y_{n-1}+1}
\\
(x_{n-1},y_{n-1}+1) & \textrm{if }T_{x_{n-1}+1,y_{n-1}} > T_{x_{n-1},y_{n-1}+1}
\end{cases}
\qquad (n\ge 1),
\]
with the initial condition $(x_0,y_0)=(0,0)$; the
box $(x_n,y_n)$ with entry $T_{x_n,y_n}$ slides into
position $(x_{n-1},y_{n-1})$ for each $n\ge1$. (The sequence
of boxes $((x_n,y_n))_{n=0}^\infty$ is called the \emph{jeu
	de taquin path} of $T$; it is shown in blue in
\cref{fig:jdt-forest}.) In words, at each step the path moves to whichever of
the two neighbours $(x_{n-1}+1,y_{n-1})$ and $(x_{n-1},y_{n-1}+1)$ holds the
smaller entry.

\item Decrease all entries by $1$.

\end{enumerate}

\end{itemize}

Below, we denote by $\mathcal{B}$ the Borel $\sigma$-algebra
on $[0,1]^\N$, and by $\Lambda$ the product of Lebesgue
measures on $[0,1]^\N$. We summarize the main results of
\cite{RomikSniady2015}:
\begin{theorem}
\label{thm:romiksniady-main}

\begin{enumerate}

\item The Plancherel measure $P$ is the push-forward of the
infinite-dimensional Lebesgue measure $\Lambda$ on
$[0,1]^\N$ under the map $\RSK(\cdot)$.

\item The jeu de taquin map $J$ is a measure-preserving
transformation on the Plancherel measure space $(\Omega,
\mathcal{F}, P)$.

\item The jeu de taquin map is intertwined with the shift map $S$
on $[0,1]^\N$ (defined by $S(z_1,z_2,\ldots)=(z_2,z_3,\ldots)$)
by the map $\RSK(\cdot)$, in the sense that the
relation
\[
J \circ \RSK = \RSK\circ S
\]
holds $\Lambda$-almost surely on $[0,1]^\N$. 

\item The RSK map $\RSK(\cdot)$ is an isomorphism between
the two measure-preserving dynamical systems $([0,1]^\N,
\mathcal{B}, \Lambda, S)$ and $(\Omega, \mathcal{F}, P, J)$.

\item The inverse map $\RSK^{-1}(\cdot)$ is given explicitly by
\begin{equation} \label{eq:inverse-rsk-explicit}
\RSK^{-1}(T) = ( \zeta(T), \zeta\circ J(T), \zeta \circ J^2(T), \ldots)
\quad \textrm{for $P$-almost every }T\in \Omega,
\end{equation}
where
\begin{equation}
\label{eq:inverse-rsk-firstcoord}
\zeta(T) =
1 - F\left(\lim_{n\to\infty} \frac{x_n-y_n}{\sqrt{T_{x_n,y_n}}} \right),
\end{equation}
where $((x_n,y_n))_{n=0}^\infty$ denotes the jeu de taquin
path associated with $T$ as in \eqref{eq:jdt-path} and
$F(u)$ is given by
\begin{equation}
\label{eq:def-F-semicircle}
F(u) = \frac{1}{2\pi} \int_{-2}^u \sqrt{4-s^2} \, ds
\qquad (-2 \le u \le 2).
\end{equation}
(This statement includes the assertion that the limit in
\eqref{eq:inverse-rsk-firstcoord} exists and takes a value in $[-2,2]$
$P$-almost surely.)

\end{enumerate}
\end{theorem}

\subsection{The jeu de taquin forest and key results} Note
that $F(\cdot)$ in \eqref{eq:def-F-semicircle} is the
cumulative distribution function of the semicircle
distribution on~$[-2,2]$. The fact that the limit in the
definition of $\zeta(T)$ exists almost surely for a
Plancherel-random IYT $T$ is equivalent to the statement
that the jeu de taquin path of $T$ is almost surely
asymptotically a straight line, with the formulas
\eqref{eq:inverse-rsk-firstcoord}--\eqref{eq:def-F-semicircle} 
being a convenient way to describe the distribution of the
random direction of this line. (See Theorem~1.1 of
\cite{RomikSniady2015} for an alternative description of the
law of this random variable.)

One surprising aspect of these results is that, while in the
finite setting the usual RSK correspondence is a
combinatorial bijection that requires both the insertion
tableau and recording tableau in order to invert the map and
recover the original sequence $z_1,\ldots,z_n$, in the
infinite setting the insertion tableau is no longer needed.
(Moreover, in the infinite setting it is not even meaningful
anymore to talk about an insertion tableau.) Furthermore,
remarkably, the recovery process for the infinite sequence
consists of iterating the JDT map $J$ and calculating the
asymptotic directions of the jeu de taquin paths arising
with each iteration. Each of these asymptotic directions
produces, via the linearizing map $u \mapsto 1-F(u)$, one
additional value in the input sequence $z_1,z_2,\ldots$.

In other words, if one ``looks'' at the tableau
$T=\RSK(\mathbf{z})$, one can ``see'' the initial value
$z_1$ of the input sequence in a relatively direct way—it is
given by $\zeta(T)$, as stated
in~\eqref{eq:inverse-rsk-explicit}. However, the inversion
formula \eqref{eq:inverse-rsk-explicit} suggests that it
ought to be difficult or impossible to ``see'' the next
value $z_2$, or any of the other values in the sequence
$\mathbf{z}$, without first calculating the tableau $J(T)$.

One of the goals of this paper is to show that this
intuition is false, in the following sense: the set of
numbers $z_1, z_2, \ldots$, with some partial information
about their ordering, \emph{can} in fact be seen directly by
looking at the tableau $T$, without applying the JDT
map~$J$. In particular, the values $z_2, z_3, \ldots$ can be
calculated from the asymptotic directions of certain tree
substructures embedded within $T$, in the same way that
$z_1$ is related to the asymptotic direction of the jeu de
taquin path.

\medskip

To make this idea precise, we introduce the notion of the
\emph{jeu de taquin forest}. Given an IYT
$T=(T_{x,y})_{x,y\ge0}$, we associate with each box $(x,y)$
of $T$ an oriented edge pointing either towards $(x+1,y)$ or
$(x,y+1)$, according to whether $T_{x+1,y} < T_{x,y+1}$ or
$T_{x+1,y} > T_{x,y+1}$. This is the same local rule that
defines the jeu de taquin path. This turns $\N_0^2$ into a
directed graph in which every box has exactly one outgoing
edge. The underlying undirected graph is acyclic: the entries
strictly increase along each directed edge, so the box with
the smallest entry on a putative cycle would have both of its
incident cycle edges outgoing, contradicting the uniqueness
of the outgoing edge. The graph is therefore a (directed)
forest. We call this the jeu de taquin forest of $T$,
and denote it by $\jdtforest(T)$; see
\cref{fig:jdt-forest}.

\begin{figure}[t]
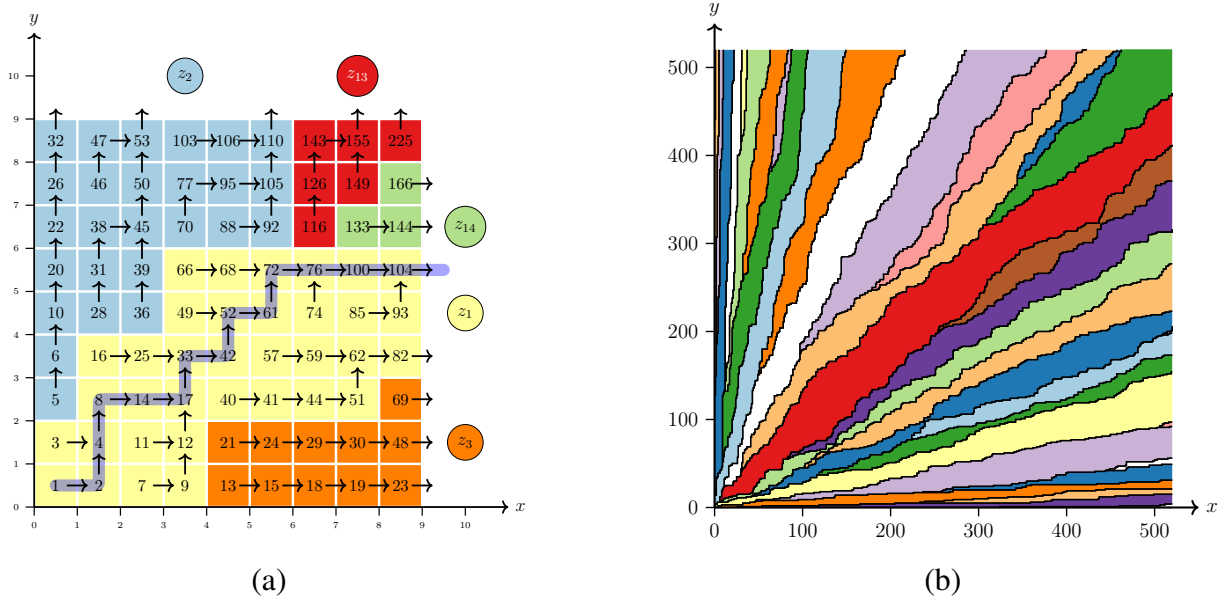

	\centering
	\subfloat[]{%
		\label{fig:jdt-forest}%
		\parbox[c][0.46\textwidth][c]{0.46\textwidth}{\centering
			\resizebox{0.46\textwidth}{!}{%
				\subfile{figures/jdt-slide-09x09-before.tex}}}%
	}\hfill
	\subfloat[]{%
		\label{fig:forest-macro}%
		\parbox[c][0.46\textwidth][c]{0.46\textwidth}{\centering
			\resizebox{0.46\textwidth}{!}{%
				\subfile{figures/FIGURE-jdt_competition_LSVK_ultra2-full.tex}}}%
	}

	\caption{The jeu de taquin forest of a
		Plan\-che\-rel-distributed IYT $T$, shown at two scales.
		\protect\subref{fig:jdt-forest} A microscopic window: a
		portion of $T$ together with its forest $\jdtforest(T)$,
		whose oriented edges are drawn as arrows. The background
		colors indicate the trees of the forest, each associated
		with an entry $z_i$ through the bijection described in
		\cref{thm:complete-inverse} (boundary annotations); the jeu
		de taquin path of $T$ is highlighted in blue.
		\protect\subref{fig:forest-macro} A macroscopic view of the
		same realization: each tree occupies a thin wedge-shaped
		region extending toward its asymptotic direction, and the
		wedges crowd together near the two coordinate axes,
		reflecting the concentration of asymptotic directions. The
		same realization recurs throughout the paper's
		illustrations; see also
		\cref{fig:jdt-slide,fig:cusp-forest,fig:competition-tree}
		and the permutation example of
		\cref{subsec:random-permutation}.}

	\label{fig:forest-overview}
\end{figure}

The JDT forest of $T$ is made up of a disjoint union of
directed trees corresponding to the connected components of
the underlying undirected forest. Since every box has an
outgoing edge, each tree contains an infinite directed path;
in particular, every tree of the forest is infinite. One of our key results
states that these trees are in a bijective correspondence with the
numbers $z_1,z_2,\ldots$ produced by the inverse RSK map
$\RSK^{-1}(T)$. The statement is as follows.

\begin{theorem}[Partial inverse of the RSK map]
	\label{thm:main-result}
	Let $T$ be a Plan\-che\-rel-distributed random IYT, and let 
	$(z_1,z_2,\ldots) = \RSK^{-1}(T)$ be the associated sequence of i.i.d.
	$U(0,1)$-distributed random variables. The following statements hold almost
	surely:
	
	\begin{enumerate}
	
	\item Each of the trees $\tree$ in $\jdtforest(T)$
possesses an asymptotic direction $\asymdirect(\tree)$ (see
\cref{subsec:asymptotic-direction} for the definition).
	
	\item We have the equality of sets
	\[
	\left\{ \asymdirect(\tree) \,:\, \tree \textrm{ is a tree in }\jdtforest(T)
	\right\}
	= \{ z_1, z_2, \ldots \}.
	\]

	\item Distinct trees have distinct asymptotic directions;
	in particular, the map $\tree \mapsto \asymdirect(\tree)$
	is a bijection between the set of trees of $\jdtforest(T)$
	and the set $\{ z_1, z_2, \ldots \}$.

	\end{enumerate}
	
\end{theorem}

\cref{thm:main-result} shows that, for a Plancherel-random
IYT $T$, the preimage $(z_1,z_2,\ldots)$ of $T$ under the
inverse RSK map is determined \emph{up to the ordering of
	the $z_n$'s} by the asymptotic directions of the trees in
the jeu de taquin forest of $T$. Moreover, some partial
information about the ordering of these numbers can also be
inferred by inspecting the forest: for example, $z_1$ is the
asymptotic direction of the unique tree in the forest
containing the box $(0,0)$ at the origin (see
\cref{lem:tree_uniform_limit}).

The full ordering can be recovered as well, at the price of
bringing back the jeu de taquin map~$J$. As we show in
\cref{sec:proofs-general}, under each application of $J$
exactly one tree of the forest --- the one containing the
origin --- dissolves into the neighbouring trees, while all
the other trees survive, only grow, and can be tracked
through the transformation. The \emph{lifetime} $n(\tree)$
of a tree $\tree$ is the number of applications of $J$ after
which $\tree$ is consumed in this way (see
\cref{sec:proofs-general} for the precise definition). The
following result upgrades \cref{thm:main-result} to a
complete description of the inverse RSK map.

\begin{theorem}[Complete inverse of the RSK map via lifetimes]
	\label{thm:complete-inverse}
	Let $T$ be a Plan\-che\-rel-distributed random IYT, and
	let $(z_1,z_2,\ldots) = \RSK^{-1}(T)$ be the associated sequence of i.i.d.
	$U(0,1)$-distributed random variables. Then, almost
	surely, the lifetime map $\tree \mapsto n(\tree)$ is a
	bijection between the set of trees of $\jdtforest(T)$ and
	the positive integers $\N$, and the numbers $z_n$ are given by
	\[
		z_n = \asymdirect(\tree_n) \qquad (n\ge1),
	\]
	where $\tree_n$ denotes the unique tree with lifetime
	$n(\tree_n) = n$.
\end{theorem}

\subsection{Related constructions of directed graphs}
Directed graph structures defined by local comparison rules on
tableau entries have appeared in related contexts. Fang \cite{Fang2015}
studied a \emph{trace forest of jeu de taquin} for finite skew tableaux
with equivalent local rules (after rotation), but focused on
Littlewood--Richardson coefficients rather than infinite tableaux or
asymptotic directions. In the study of last passage percolation,
the notion of \emph{geodesic coalescence} generates similar tree structures
\cite{CatorPimentel2013,FerrariPimentel2005,GeorgiouRassoulAghaSeppaelaeinen2017}.

\subsection{Structure of the paper} In the process of
proving \cref{thm:main-result,thm:complete-inverse}, we will prove many more
interesting facts that reveal subtle aspects of the
structure of jeu de taquin forests and the way they interact
with the jeu de taquin map, and are of independent interest.
Thus, the paper does not have a clearly defined ``main
result''. 

The outline is as follows: in
\cref{sec:trajectories-trees-cusps} we develop the basic combinatorics
of jeu de taquin trajectories, trees, and their cusps, and in
\cref{sec:proofs-general} we study how the forest evolves under the jeu
de taquin map; both sections concern an arbitrary IYT and use no
probabilistic (or deterministic) assumptions.
\cref{sec:inverse-rsk-details} recalls some standard facts from the
theory of Plancherel measure, and recasts them in geometric
language by defining a coordinate system that relates in a
useful way to the geometry of jeu de taquin paths --- a coordinate system which
we label \emph{RSK-polar coordinates}.
\cref{sec:proofs-plancherel-normal} proves additional results on jeu
de taquin forests that are specific to a deterministic class
of IYTs we call \emph{Plancherel-normal}, which can be
interpreted as ``typical'' random tableaux sampled from
Plancherel measure.
Theorems~\ref{thm:main-result}~and~\ref{thm:complete-inverse} are proved in
\cref{subsec:proof-thm-main-result}.

The final section, \cref{sec:final-remarks}, is more speculative in nature, and
contains a detailed discussion of several open problems and conjectures that
arise from the ideas introduced in the current paper. This makes the point that
jeu de taquin forests are an important new direction for the study of infinite
Young tableaux, and that there is much more still to study and understand about
them.

\section{Jeu de taquin trajectories, trees, and cusps}
\label{sec:trajectories-trees-cusps}

In this section we work in a purely combinatorial setting, without any
probability and
without invoking the jeu de taquin map. Throughout the section, $T$ denotes a
fixed (deterministic) IYT.

\subsection{Trajectories and trees}
\label{subsec:trajectories-trees}

We generalize the jeu de taquin trajectories studied in
\cite{RomikSniady2015}, which focused on paths starting from the origin,
by considering trajectories starting at arbitrary initial boxes.

\begin{definition}[Jeu de taquin trajectory]
    \label{def:jdt-trajectory}
    Let $T$ be an IYT and $(x,y) \in \N_0^2$
    a box. The \emph{jeu de taquin trajectory} $\jdtpath^T_{x,y}$ is the
    unique infinite path in the jeu de taquin forest of $T$ that starts
    at $(x,y)$ and follows the directed edges.
\end{definition}

Since the jeu de taquin forest consists of trees with edges pointing
away from the origin, each trajectory extends toward infinity. The
trajectory from the origin, $\jdtpath^T_{0,0}$, coincides with the
classical jeu de taquin path studied in \cite{RomikSniady2015}.

Trajectories are governed by an elementary confluence property.

\begin{lemma}[Trajectories in a common tree coincide]
\label{lem:trajectories-merge}
Two boxes lie in the same tree of $\jdtforest(T)$ if and only if their
jeu de taquin trajectories eventually coincide.
\end{lemma}

\begin{proof}
If two trajectories share a box then, every box having a single outgoing
edge, they coincide from that box onward; in particular their starting
boxes lie in the same tree. Conversely, ``eventually coinciding'' is a
transitive relation on boxes, so it suffices to verify it for two boxes
$u,v$ joined by an edge of $\jdtforest(T)$. That edge is oriented one
way, say from $u$ to $v$; then $v$ is the second box of the trajectory of
$u$, so the trajectory of $u$ coincides with that of $v$ from $v$ onward.
Applying transitivity along the finite path joining any two boxes of a
common tree, their trajectories eventually coincide.
\end{proof}

\subsection{The lazy parametrization}

We parametrize trajectories using the \emph{lazy parametrization} of
\cite{RomikSniady2015}, adapted to our slightly more general setting. Let
$t_0 = T_{x,y}$ be the entry at the starting box $(x,y)$; since entries
increase along the trajectory, $t_0$ is the smallest entry occurring on
it. For each integer $t \geq 0$ we define $\jdtpath^T_{x,y}[t] \in \N_0^2$
to be the last box along the trajectory whose entry is at most
$\max(t, t_0)$: the box rests at its starting position $(x,y)$ until time
$t_0$, and thereafter advances one step each time $t$ reaches the entry of
the next box to the north or east along the path. We write the parameter
in square brackets to distinguish the lazy parametrization
$t \mapsto \jdtpath^T_{x,y}[t]$ from the trajectory $\jdtpath^T_{x,y}$
itself, defined in \cref{def:jdt-trajectory} as its underlying set of
boxes.

The point of the parametrization is that it lets us compare several
trajectories at a common \emph{time} $t$ by recording their positions on
the boundary of the Young diagram
\[
\lambda^{(t)} = \{(x',y') : T_{x',y'} \leq t\} .
\]

\subsection{Comparison of trajectories}
\label{subsec:trajectory-comparison}

\subsubsection{Partial orders on the quarter-plane}

We introduce a partial order that will be essential for comparing 
trajectories and their asymptotic directions.

Define a partial order on $\N_0^2 \setminus \{(0,0)\}$ by
\begin{equation}
	\label{eq:partial_order}
	(x_1,y_1) \preceq (x_2,y_2)
	\iff x_1 \geq x_2 \text{ and } y_1 \leq y_2.
\end{equation}
This order has a natural geometric interpretation: $(x_1,y_1) \preceq 
(x_2,y_2)$ means that $(x_1,y_1)$ is weakly to the northwest of 
$(x_2,y_2)$.

We will also use a separate coordinatewise
order on $\N_0^2$, defined by
\[
	(x_1,y_1) \trianglelefteq (x_2,y_2)
	\iff x_1 \leq x_2 \text{ and } y_1 \leq y_2,
\]
under which $(x_2,y_2)$ lies weakly to the northeast of $(x_1,y_1)$.

\subsubsection{Convex corners}

A box in a Young diagram is called a \emph{convex corner}
(also called a \emph{removable box} or \emph{inner corner})
if it can be removed while still leaving a valid Young
diagram. The next lemma states two elementary properties of
convex corners; since the claims are geometrically obvious,
we do not include a proof.

\begin{lemma}[Convex corners]
	\label{lem:convex-corners}
	\
	
	\begin{enumerate}[label=(\roman*)]
		\item 
		\label{item:tc-2}
		Any two convex corners of a Young diagram are comparable with 
		respect to the partial order $\preceq$ defined in 
		\eqref{eq:partial_order}.
		
		\item 
		\label{item:tc-3}
		If $(x_1,y_1)$ and $(x_2,y_2)$ are two distinct
		convex corners of a Young diagram, then their Russian
		coordinates $u_i = x_i - y_i$ differ by at least $2$:
		$|u_1 - u_2| \geq 2$.
	\end{enumerate}
\end{lemma}

The lazy parametrization reveals important geometric properties of 
trajectories, particularly their relationship to convex corners of 
Young diagrams.

\begin{lemma}[Trajectories and convex corners]
	\label{lem:trajs-convex-corners}
		For any trajectory $\jdtpath^T_{x,y}$ and any $t \geq T_{x,y}$,
		the position $\jdtpath^T_{x,y}[t]$ is a convex corner of the
		Young diagram~$\lambda^{(t)}$.
\end{lemma}

\begin{proof}
	Let $(x',y') = \jdtpath^T_{x,y}[t]$.
	By definition of the lazy parametrization, $(x',y')$ is the
	last box along the trajectory with entry at most $t$, hence
	$(x',y') \in \lambda^{(t)}$. The next box in the jeu de
	taquin path is either the eastern neighbor $(x'+1,y')$ or the
	northern neighbor $(x',y'+1)$, whichever contains the smaller
	entry in $T$. It follows that the entry in the next box is
	\[
	m = \min\big\{T_{x'+1,y'},\ T_{x',y'+1}\big\}.
	\]
	Since $(x',y')$ is the \emph{last} box with entry at most $t$,
	the next entry on the trajectory satisfies $m > t$. Therefore,
	both $T_{x'+1,y'} > t$ and $T_{x',y'+1} > t$, which means both
	$(x'+1,y')$ and $(x',y'+1)$ lie outside $\lambda^{(t)}$; that
	is an equivalent description of what it means for $(x',y')$ to be a convex
	corner.
\end{proof}

\subsubsection{Order preservation under the lazy dynamics}

The next lemma states the useful fact that the partial order
between two lazy trajectories, once present, is never lost.

\begin{lemma}[Order preservation]
	\label{lem:order-preservation}
	Let $T$ be an IYT and let $\jdtpath^T_a, \jdtpath^T_b$ be the lazy
	trajectories starting at boxes $a, b \in \N_0^2$. If
	$\jdtpath^T_a[s_0] \preceq \jdtpath^T_b[s_0]$ for some integer
	$s_0 \geq 0$, then
	\[
		\jdtpath^T_a[t] \preceq \jdtpath^T_b[t]
		\qquad \text{for all } t \geq s_0 .
	\]
	Moreover, once the two positions coincide they remain equal at all
	later times.
\end{lemma}

\begin{proof}
	Write $p[t] = \jdtpath^T_a[t]$ and $q[t] = \jdtpath^T_b[t]$. It
	suffices to prove the one-step claim that $p[t] \preceq q[t]$
	implies $p[t+1] \preceq q[t+1]$; the assertion then follows by
	induction on $t \geq s_0$. Two facts are used repeatedly: along any
	trajectory each coordinate is non-decreasing in $t$ (a step
	increases $x$ or $y$ by one, and a waiting box does not move); and
	$\lambda^{(t)}$ is a Young diagram, hence \emph{down-closed} --- if
	$(x',y') \in \lambda^{(t)}$ then every box $\trianglelefteq (x',y')$
	also lies in $\lambda^{(t)}$.

	Assume $p[t] \preceq q[t]$.

	\medskip

	\emph{Case 1: at least one of the two boxes is still waiting at
	time~$t$}, say $b$, so that $q[t] = b \notin \lambda^{(t)}$. From
	$p[t] \preceq b$ we have $x(p[t]) \geq x_b$ and $y(p[t]) \leq y_b$.
	If $p[t]$ were active with $y(p[t]) = y_b$, then $b \trianglelefteq
	p[t]$, and down-closedness would force $b \in \lambda^{(t)}$, a
	contradiction; hence either $p[t]$ also waits or $y(p[t]) < y_b$. In
	either situation a single step keeps $x(p[t+1]) \geq x_b$ and
	$y(p[t+1]) \leq y_b$, while $q[t+1] = b$ (a box does not move on the
	step at which it is born). Thus $p[t+1] \preceq q[t+1]$. The case
	where $a$ waits is symmetric, with the roles of the two coordinates
	exchanged.
	
	\medskip

	\emph{Case 2: both boxes are active at time~$t$.} Then $p[t], q[t]$
	are convex corners of $\lambda^{(t)}$
	(\cref{lem:trajs-convex-corners}), as are $p[t+1], q[t+1]$ of
	$\lambda^{(t+1)}$. If $p[t] = q[t]$, the trajectories have merged and
	thereafter follow identical rules, so they remain equal. Otherwise
	$p[t] \preceq q[t]$ yields $u(p[t]) \geq u(q[t])$ for the Russian
	coordinate $u = x - y$, with gap at least~$2$
	(\cref{lem:convex-corners}\ref{item:tc-3}). Each coordinate, and
	hence $u$, changes by at most~$1$ along a trajectory per unit time,
	so
	\[
		u(p[t+1]) - u(q[t+1])
		\geq \big( u(p[t]) - u(q[t]) \big) - 2 \geq 0 .
	\]
	As $p[t+1]$ and $q[t+1]$ are comparable convex corners
	(\cref{lem:convex-corners}\ref{item:tc-2}), the inequality
	$u(p[t+1]) \geq u(q[t+1])$ gives $p[t+1] \preceq q[t+1]$, with
	equality precisely when the two positions coincide.
\end{proof}

\subsection{The cusp of a tree}
\label{subsec:cusp}

Fix an IYT $T$ and a tree $\tree$ of its jeu de taquin forest
$\jdtforest(T)$. The first coordinates of the boxes of $\tree$ form a
non-empty subset of $\N_0$ and so attain a minimum
$x_{\min} = \min\{ x : (x,y) \in \tree \}$, realised by one or more
\emph{leftmost} boxes of $\tree$; likewise the second coordinates attain
a minimum $y_{\min}$, realised by one or more \emph{lowest} boxes. These
extremal boxes need not be unique.

\begin{definition}[Cusp of a tree]
	\label{def:cusp}
	The \emph{cusp} of the tree $\tree$ is the box
	\[
		\cusp(\tree) := (x_{\min}, y_{\min}) \in \N_0^2 ,
	\]
	formed from the minimal first and second coordinates occurring in~$\tree$.
\end{definition}

A priori the cusp is merely a point of the quarter-plane, assembled from
two separate coordinate minima that need not be attained at a common box.
The main result of this subsection is that it is in fact a box of
$\tree$---indeed its unique minimal box in the coordinatewise order.

\begin{theorem}[The cusp is the minimal box of the tree]
	\label{thm:cusp-in-tree}
	For every IYT $T$ and every tree $\tree$ of $\jdtforest(T)$, the cusp
	$\cusp(\tree)$ belongs to $\tree$. It is the unique
	$\trianglelefteq$-minimal box of $\tree$; that is,
	$\cusp(\tree) \trianglelefteq (x,y)$ for every $(x,y) \in \tree$.
\end{theorem}

\begin{proof}
	Let $\ell = (x_{\min}, y_\ell)$ be a leftmost box, $m = (x_m,
	y_{\min})$ a lowest box, and $c = (x_{\min}, y_{\min}) = \cusp(\tree)$.
	Directly from the definition of $\preceq$,
	\[
		m \preceq c \preceq \ell ,
	\]
	since $x_m \geq x_{\min}$ and $y_{\min} \leq y_\ell$, the remaining
	coordinates coinciding.

	Consider the lazy trajectories $\jdtpath^T_m$, $\jdtpath^T_c$,
	$\jdtpath^T_\ell$. All tableau entries are positive, so at time
	$t = 0$ each rests at its starting box; hence $\jdtpath^T_m[0] = m
	\preceq c = \jdtpath^T_c[0] \preceq \ell = \jdtpath^T_\ell[0]$. By
	\cref{lem:order-preservation} (applied to each of the two pairs),
	\begin{equation}
		\label{eq:cusp-chain}
		\jdtpath^T_m[t] \preceq \jdtpath^T_c[t] \preceq \jdtpath^T_\ell[t]
		\qquad \text{for all } t \geq 0 .
	\end{equation}

	\medskip
	The boxes $m$ and $\ell$ both lie in $\tree$. Two boxes of the same
	tree have trajectories that eventually merge (by
	\cref{lem:trajectories-merge}),
	so
	$\jdtpath^T_m[t] = \jdtpath^T_\ell[t]$ for all sufficiently large
	$t$. For such $t$, \eqref{eq:cusp-chain} squeezes the middle term,
	$\jdtpath^T_m[t] \preceq \jdtpath^T_c[t] \preceq \jdtpath^T_m[t]$,
	whence $\jdtpath^T_c[t] = \jdtpath^T_m[t]$ by antisymmetry of
	$\preceq$. Thus the trajectory of $c$ eventually coincides with that
	of $m$, and in particular meets $\tree$. Since the trajectory from
	any box follows the directed edges of $\jdtforest(T)$ and therefore
	remains within the tree containing that box, we conclude $c \in
	\tree$.

	\medskip
	Finally $c = (x_{\min}, y_{\min}) \trianglelefteq (x,y)$ for every
	$(x,y) \in \tree$ by minimality of $x_{\min}$ and $y_{\min}$, so $c$
	is the unique $\trianglelefteq$-minimal box of $\tree$.
\end{proof}

Since the entries of $T$ increase along rows and columns,
they are monotone with respect to the coordinatewise order
$\trianglelefteq$. \Cref{thm:cusp-in-tree} therefore has the
following immediate consequence.

\begin{corollary}[The cusp realizes the minimal entry]
	\label{cor:cusp-min-entry}
	For every tree $\tree$ of $\jdtforest(T)$, the smallest
	entry occurring in $\tree$,
	\[
		\min \tree := \min \big\{ T_{x,y} : (x,y) \in \tree \big\},
	\]
	is attained at the cusp: $\min \tree = T_{\cusp(\tree)}$.
\end{corollary}

\section{Evolution of forests under jeu de taquin dynamics}
\label{sec:proofs-general}

In this section we bring in the jeu de taquin map $J$ and analyse how the forest
$\jdtforest(T)$ evolves under it. As in the previous section, $T$ still refers
to an arbitrary fixed (deterministic) IYT.

\subsection{Evolution under a single iteration of $J$}
	
When the JDT map $J$ is applied to $T$, the output tableau
$J(T)$ has its own jeu de taquin forest structure
$\jdtforest(J(T))$. Our first result sheds useful light on
the relationship between the two forests $\jdtforest(T)$ and
$\jdtforest(J(T))$. Here, we distinguish between the
\emph{tree at the origin}, which is the unique tree in
$\jdtforest(T)$ containing the origin box $(0,0)$, and all
other trees in the forest, which we refer to as
\emph{peripheral trees}. \Cref{fig:jdt-slide} illustrates a
single application of $J$ and its effect on the forest.

\begin{figure}[t]
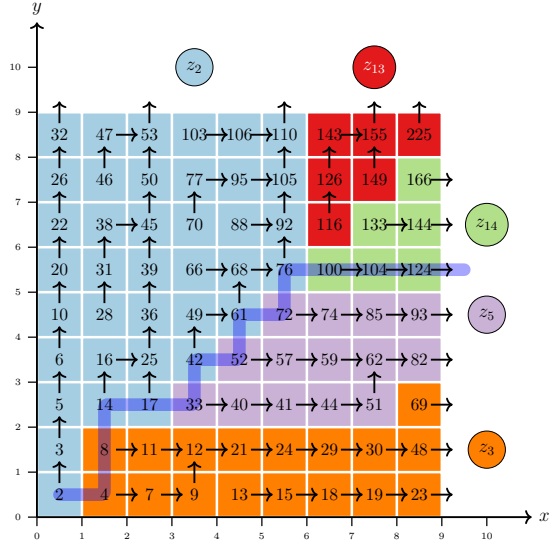

	\centering
	\resizebox{0.48\textwidth}{!}{%
		\subfile{figures/jdt-slide-09x09-after.tex}}

	\caption{After applying the sliding part of the jeu de taquin map
		$J$ to the tableau of \cref{fig:jdt-forest}, before decreasing
		all entries by $1$, the result is the tableau shown above.}
	\label{fig:jdt-slide}
\end{figure}

\begin{proposition}[Stability of peripheral trees]
	\label{thm:jdt-forest-evolution-weak}
	Let $T$ be an IYT.
	There exists a canonical injection
	\begin{equation}  \label{eq:def-iota}
	    \iota\colon \big\{\text{peripheral trees in }\jdtforest(T)\big\} 
	\to \big\{\text{all trees in }\jdtforest(\jdt(T))\big\}
	\end{equation}
	uniquely characterized by the following property: for each 
	peripheral tree $\tree$ in $\jdtforest(T)$, $\tree$ is 
	a subtree of its image $\iota(\tree)$ under the natural 
	inclusion of graphs.
\end{proposition}

\begin{proof}
	The map $\jdt$ operates by sliding
	boxes along the jeu de taquin path of $T$
	starting at the origin. By definition of the jeu de taquin
	forest, this path lies entirely within the tree at the origin. 
	Therefore, only boxes in the tree at the origin are moved by 
	the transformation; all other boxes remain stationary.

	\medskip
	We now claim that for any box $(x,y)$ in a peripheral tree of $T$,
	its outgoing edge in the JDT forest structure remains unchanged 
	after the transformation $J$. Recall that the outgoing edge from 
	$(x,y)$ points to the smaller of its two neighbors: either to 
	the north $(x, y+1)$ or to the east $(x+1, y)$, whichever has 
	the smaller entry. If neither neighbor lies on the jeu de
	taquin path, then both entries remain unchanged (apart from a
	trivial global shift of values) and there is nothing to prove.
	We therefore consider the case where at least one of these neighbors
	lies on the jeu de taquin path.
	
	Suppose the outgoing edge from $(x,y)$ points to $(x,
y+1)$, the northern neighbor. This means $T_{x,y+1} <
T_{x+1,y}$. Since $(x,y)$ belongs to a peripheral tree
rather than the tree at the origin, neither $(x,y)$ nor its
northern neighbor $(x,y+1)$ can lie on the jeu de taquin
path. Therefore, the eastern neighbor $(x+1, y)$ must lie
on the jeu de taquin path, and its entry will be modified
by sliding part of the jeu de taquin operation. However,
jeu de taquin sliding replaces each entry along the path
with a strictly larger value. (This refers just to the
sliding step before the final step in the definition of the
map $J$ that consists of decrementing all the entries by 1,
which preserves all order relations.) Consequently, after
the transformation, the inequality $\jdt(T)_{x,y+1} <
\jdt(T)_{x+1,y}$ is preserved, and the outgoing edge from
$(x,y)$ continues to point north.
	
The case where the outgoing edge points east is symmetric.

\medskip Since the outgoing edges from all boxes in
peripheral trees of $\jdtforest(T)$ are preserved as edges
in $\jdtforest(\jdt(T))$, each peripheral tree $\tree$ in
$\jdtforest(T)$ embeds as a subtree of some (possibly
larger) tree in $\jdtforest(\jdt(T))$, which we denote
$\iota(\tree)$.

Moreover, we can identify the vertices of $\iota(\tree)$ as
consisting of the original boxes of $\tree$, together with
those boxes $(x,y)$ that belong to the tree at the origin in
$\jdtforest(T)$ and whose oriented paths in the new forest
$\jdtforest(\jdt(T))$ eventually reach a box from $\tree$
(and remain in $\tree$ thereafter).

The map $\iota$ is an injection. Indeed, let $\tree, \tree'$
be two distinct peripheral trees of $T$ and let $u \in \tree$
and $u' \in \tree'$ be boxes. Since the outgoing edges of the
boxes of peripheral trees are preserved, the trajectory of
$u$ in $\jdtforest(\jdt(T))$ stays within $\tree$, and the
trajectory of $u'$ stays within $\tree'$; the sets $\tree$
and $\tree'$ being disjoint, the two trajectories never meet,
so by \cref{lem:trajectories-merge} (applied to the tableau
$\jdt(T)$) the boxes $u$ and $u'$ lie in distinct trees of
$\jdtforest(\jdt(T))$. Hence
$\iota(\tree) \neq \iota(\tree')$.

We constructed the map $\iota$ and verified the properties
it was claimed to satisfy. It is also easy to confirm that
$\iota$ is the \emph{unique} map with these properties:
since each $\tree$ is a subtree of its image $\iota(\tree)$
under the natural inclusion of graphs, and the trees of
$\jdtforest(\jdt(T))$ are disjoint, $\iota(\tree)$ is in
fact the unique tree in $\jdtforest(\jdt(T))$ that contains
$\tree$ as a subtree, so this property indeed characterizes
$\iota(\tree)$ uniquely.
\end{proof}

\subsection{Finite lifetime of jeu de taquin trees}

We now consider an extended scenario in which the JDT map
$J$ is applied iteratively. $T$ will again denote a fixed
deterministic IYT.
For each positive integer $n$, we define
\[ T_n = \jdt^{n-1}(T), \]
the tableau obtained by applying the jeu de taquin transformation
$n-1$ times to $T$. Thus $T_1 = T$, $T_2 = \jdt(T)$,
$T_3 = \jdt^2(T)$, and so forth. Each of the tableaux $T_n$ has its
own forest structure $\jdtforest(T_n)$, tree at the origin, and set
of peripheral trees. For each $n \geq 1$, denote the tree
at the origin of $\jdtforest(T_n)$ by $\origin_n$. 
For each $n\ge1$, denote by $\gamma_n$ the jeu de taquin path associated with
$T_n$.

The injection $\iota$ defined in \cref{thm:jdt-forest-evolution-weak} allows us
to track
peripheral trees through successive applications of the jeu de
taquin transformation. 

\begin{proposition}[Every tree is eventually consumed]
	\label{prop:finite-time-origin} For any tree $\tree$ in
	$\jdtforest(T)$, there exists a finite number $n = n(\tree)
	\in \{1,2,\ldots\}$ such that after $n-1$ applications of
	the jeu de taquin transformation, the tree $\tree$ (tracked
	through the injections from
	\cref{thm:jdt-forest-evolution-weak}) becomes $\origin_n$,
	the tree at the origin in $T_n$. The number $n(\tree)$
	satisfies the upper bound
	\begin{equation}  \label{eq:tree-lifetime-upperbound}
	n(\tree) \le \min_{(x,y) \in \tree} \Big( T_{x,y} - x - y - xy \Big).
	\end{equation}
\end{proposition}

We call $n(\tree)$ the \emph{lifetime} of the tree $\tree$.
This number is the last value $n$ for which $\tree$ can be
tracked using the injection maps $\iota$ over $n-1$
iterations of the map $J$ to be naturally associated with a
tree of $T_n$. After one additional iteration of the jeu de
taquin map applied to $T_n$, the tree is consumed.

\begin{proof}
Let $(x, y)$ be a box of $\tree$. The entry $T_{x, y}$ must satisfy
\begin{equation}
\label{eq:Txy-lower-bound}
T_{x, y} \ge (x+1)(y+1),
\end{equation}
since $T_{0,0}=1$ and the entries of $T$ are monotonically
increasing along rows and columns. If there is equality in
\eqref{eq:Txy-lower-bound}, then $(x,y)$ must lie on the jeu
de taquin path $\gamma_1$ of $T$, since otherwise, after
applying the jeu de taquin map, the tableau $J(T)$ will have
entry $J(T)_{x,y} = (x+1)(y+1)-1$ at $(x,y)$, which violates
the inequality \eqref{eq:Txy-lower-bound} with respect to
$J(T)$. Since $\gamma_1$ is contained in $\origin_1$, this
means that $n(\tree) = 1$.

By similar reasoning, in the general case where
\eqref{eq:Txy-lower-bound} may be a strict inequality, since
each successive application of the jeu de taquin map in
which $(x,y)$ does not lie on the jeu de taquin path has the
effect of decreasing the entry at $(x,y)$ by $1$, it follows
that after some number $n-1$ of successive applications of
the jeu de taquin map that is at most $T_{x,y}- (x+1)(y+1)$,
we will arrive at a tableau in which $(x,y)$ lies on the
jeu de taquin path $\gamma_n$. This means that the tree
$\tree$, tracked through the injections from
\cref{thm:jdt-forest-evolution-weak}, has become
$\origin_n$, the tree at the origin in $T_n$. This
establishes that the tree lifetime $n(\tree)$ is
well-defined, finite, and satisfies the upper bound
$n(\tree) \le T_{x,y}-(x+1)(y+1) + 1 = T_{x,y}-x-y-xy$.
Since this bound holds for \emph{any} box $(x,y)$ in
$\tree$, taking the minimum over all boxes gives
\eqref{eq:tree-lifetime-upperbound}.
\end{proof}

\subsection{The maximal extension of a tree}

Another useful notion is the \emph{maximal extension} of a
tree $\tree$. This object captures the full extent of
$\tree$ by following it forward through the sequence of JDT
map iterations until it arrives at the origin.

\begin{definition}[Maximal extension]
	\label{def:tree-max}
	Given an IYT $T$ and a tree $\tree$ in $\jdtforest(T)$, the
	\emph{maximal extension} 
	of $\tree$ is defined by
	\[ \tree_{\max} = \origin_n, \]
	where as before $n=n(\tree)$ is the lifetime of $\tree$ and
	$\origin_n$ is the tree at the origin in $\jdtforest(T_n)$.
\end{definition}

As we track $\tree$ 
through the sequence of injections $\iota$ as it evolves under 
iterations of $J$, ultimately becoming the tree at the 
origin in $\jdtforest(T_n)$, each step only makes the tree grow, 
so in particular we have $\tree \subseteq 
\tree_{\max}$. The maximal extension can thus be viewed as the closure 
of $\tree$ with respect to jeu de taquin dynamics: it represents 
the maximal set of boxes that will eventually be absorbed into 
$\tree$ as it evolves toward the origin.

\begin{remark}[Maximal extension in transformed tableaux]
	\label{rem:lifetime-after-jdt} 
	If $\tree$ is a tree in the
	jeu de taquin forest of the transformed tableau $T_k =
	\jdt^k(T)$ for some $k \geq 1$, let $n(\tree)$ denote the
	lifetime of $\tree$ relative to $T_k$, and let
	$\tree_{\max}$ denote its maximal extension. Then
	$\tree_{\max} = \origin_m$ where $m = n(\tree) + k\geq k$.
\end{remark}

\subsection{Absorption of boxes into peripheral trees}

The following lemma locates the boxes absorbed from the tree at
the origin into peripheral trees during a single jeu de taquin
transformation (cf.\ \cref{fig:jdt-slide}).

\begin{lemma}[Absorbed boxes]
	\label{prop:absorbed-boxes}
	Let $T$ be an IYT. For each 
	peripheral tree $\tree$ in~$T$, the set of boxes absorbed 
	during the first jeu de taquin transformation
	\[ \iota(\tree) \setminus \tree \]
	is contained in $\origin_1 \cap \origin_n$ for some 
	$n \geq 2$.
\end{lemma}

\begin{proof}
	Let $\tree$ be a peripheral tree in $T$. The boxes in 
	$\iota(\tree) \setminus \tree$ are precisely those absorbed 
	by $\tree$ during the first jeu de taquin transformation. 
	By the explicit description of the boxes of $\iota(\tree)$ 
	in the proof of \cref{thm:jdt-forest-evolution-weak}, these boxes must 
	originate from the tree at the origin in $T$. Therefore,
	\begin{equation}  \label{eq:iotaTminusT-subsetO1}
	    \iota(\tree) \setminus \tree \subseteq \origin_1.
	\end{equation}
	
	On the other hand, after the first transformation, $\iota(\tree)$ 
	is a tree in $T_2 = \jdt(T)$. By \cref{rem:lifetime-after-jdt}, 
	the maximal extension of $\iota(\tree)$ is the tree at the 
	origin in some tableau $T_n$ where $n \geq 2$. That is,
	\[ [\iota(\tree)]_{\max} = \origin_n \]
	for some $n \geq 2$. Since $\iota(\tree) \subseteq 
	[\iota(\tree)]_{\max}$, we have
	\begin{equation} \label{eq:iotaTminusT-subsetOn}
	\iota(\tree) \setminus \tree \subseteq \iota(\tree) 
	\subseteq \origin_n.
	\end{equation}
Combining \eqref{eq:iotaTminusT-subsetO1} and
\eqref{eq:iotaTminusT-subsetOn} gives the claim.
\end{proof}

\section{Geometric notions for infinite Young tableaux} 
\label{sec:inverse-rsk-details}

We now develop some useful geometric notions that will
feature in our analysis of Plancherel-random IYTs in
\cref{sec:proofs-plancherel-normal}.

\subsection{The LSVK limit shape and RSK-polar coordinates}

\label{sec:LSVK-curve}

The map $F(\cdot)$ defined in \eqref{eq:def-F-semicircle} is
closely related to the famous curve established by
Logan--Shepp \cite{LoganShepp1977} and, independently,
Vershik--Kerov \cite{VershikKerov1977}, as the limit shape
of random Young diagrams under the Plancherel measure.
Consider the random Young diagram formed by the boxes
labeled $1, 2, \ldots, n$ in a Plancherel-distributed random
infinite tableau. (That is, if one thinks in terms of the
RSK map, this is the Young diagram $\lambda^{(n)}$ described
in \cref{sec:intro}). After rescaling by $\sqrt{n}$, the
boundary of these diagrams converges almost surely to a
deterministic curve as $n \to \infty$.
This limiting curve is the
\emph{Logan--Shepp--Vershik--Kerov (LSVK) limit shape}.

The LSVK curve is naturally expressed in the \emph{Russian
	coordinates} $(u,v)$,  \cite[p.~35]{Romik2015} related to
the $(x,y)$ coordinates (sometimes called \emph{French
	coordinates} in this context) by
\[
u = x - y, \qquad v = x + y.
\]
In those coordinates, the LSVK curve takes the explicit 
form
\[
v = \Omega(u) = \frac{2}{\pi} \left( u \sin^{-1}\left(\frac{u}{2}
\right) + \sqrt{4 - u^2} \right), \quad |u| \leq 2.
\]

The LSVK curve can be parametrized in a way that reflects its connection to 
the RSK correspondence by making use of the 
cumulative distribution function $F(\cdot)$ from \eqref{eq:def-F-semicircle}.
For $z \in [0,1]$, define 
\begin{equation}
\label{eq:oneminusz-quantile}
\scq_z = F^{-1}(1-z) 
\in [-2,2],
\end{equation}
the $(1-z)$-quantile of the semicircle distribution on $[-2,2]$. We introduce
the \emph{RSK 
	trigonometric functions}:
\begin{align*}
	\RSKcos z &= \frac{\Omega(\scq_{z}) + \scq_{z} }{2}, 
	\\[0.5em]
	\RSKsin z &= \frac{\Omega(\scq_{z}) - \scq_{z}}{2}. 
\end{align*}
As is immediate to see from the definitions, the map 
\begin{equation}
	\label{eq:LSVK-parametrization}
	z \mapsto \bigl( \operatorname{RSKcos} z, 
	\operatorname{RSKsin} z \bigr) \qquad (z \in [0,1])
\end{equation}
is a parametrization of the LSVK curve, tracing it in the
French coordinates from $(2,0)$ at $z=0$ to $(0,2)$ at
$z=1$.

It is helpful to generalize this parametrization to a
two-dimensional curvilinear coordinate system on the
positive quadrant, which we call the \emph{RSK-polar
	coordinates}, and denote by~$(\rho, \alpha)$. These
coordinates relate to the $(u,v)$ and $(x,y)$ coordinate
systems by
\begin{equation}
\Big\{
\begin{array}{rcl}
x &=& \rho \RSKcos \alpha,
\\
y &=& \rho \RSKsin \alpha,
\end{array}
\quad \textrm{or, equivalently,} \quad
\Big\{
\begin{array}{rcl}
u &=& \rho \, \scq_\alpha,
\\
v &=& \rho \, \Omega(\scq_\alpha).
\end{array}
\label{eq:def-rsk-polar}
\end{equation}
We call $\rho$ the \emph{RSK-norm} and $\alpha$ the
\emph{RSK-angle} of $(x,y)$, respectively, and denote
\[
\rho = \rho(x,y), \qquad \alpha = \alpha(x, y).
\]
It is easy to see that the
relations~\eqref{eq:def-rsk-polar} associate to each $(x,y)$
in the positive quadrant, except for the origin, a unique
pair $(\rho, \alpha)$ with $\rho \ge 0$, $\alpha \in [0,1]$,
and $(\rho,\alpha) \neq (0,0)$. For the case $(x,y)=(0,0)$
we have $\rho=0$ but the $\alpha$ coordinate is undefined.

The next two lemmas are not used anywhere in the paper, but have not been
published before, and relate in an interesting way to existing ideas appearing
in the literature --- see
Remarks~\ref{remark:ode-kerov-characterization}~and~\ref{remark:folklore-characterization}
below --- so we think they are worth including here.

\begin{lemma}
The Jacobian $\frac{\partial(x,y)}{\partial(\rho,\alpha)}$
of the RSK-polar coordinate map $(\rho,\alpha) \mapsto
(x,y)$ is given by
\[
\frac{\partial(x,y)}{\partial(\rho,\alpha)} = 2 \rho.
\]
\end{lemma}

\begin{proof}
Observe that the function $\Omega(u)$ satisfies the relation
\begin{equation}
\label{eq:omega-ode}
\Omega(u) - u \Omega'(u) = \frac{2}{\pi} \sqrt{4-u^2},
\end{equation}
and that
$\frac{d \scq_\alpha}{d\alpha} = -\left(\frac{1}{2\pi} \sqrt{4-\scq_\alpha^2} \right)^{-1}$, 
by \eqref{eq:def-F-semicircle} and \eqref{eq:oneminusz-quantile}.
Therefore
\begin{align*}
\frac{\partial(u,v)}{\partial(\rho,\alpha)} &= \det 
\begin{pmatrix} 
\frac{\partial u}{\partial \rho} & \frac{\partial u}{\partial \alpha} \\[3pt]
\frac{\partial v}{\partial \rho} & \frac{\partial v}{\partial \alpha}
\end{pmatrix}
= \det \begin{pmatrix} 
\scq_\alpha & \phantom{\Omega'(\scq_\alpha)} \rho \frac{d \scq_\alpha}{d \alpha}
\\[3pt]
\Omega(\scq_\alpha) & \Omega'(\scq_\alpha) \, \rho
\frac{d \scq_\alpha}{d \alpha} 
\end{pmatrix}
= \rho \frac{d \scq_\alpha}{d \alpha} 
\det \begin{pmatrix} \scq_\alpha & 1 \\ \Omega(\scq_\alpha) &
\Omega'(\scq_\alpha) 
\end{pmatrix}
\\ &= \rho \frac{-2\pi}{ \sqrt{4-\scq_\alpha^2}} 
\left( \scq_\alpha \Omega'(\scq_\alpha) - \Omega(\scq_\alpha) \right) 
= 4\rho.
\end{align*}
It follows that $\frac{\partial(x,y)}{\partial(\rho,\alpha)} = 
\frac{\partial(x,y)}{\partial(u,v)} \cdot
\frac{\partial(u,v)}{\partial(\rho,\alpha)} =  
\frac12\cdot 4\rho = 2\rho$, as claimed.
\end{proof}

The parametrization \eqref{eq:LSVK-parametrization} has a natural geometric
interpretation, illustrated in~\cref{fig:lsvk-a},
and described in the following lemma.

\begin{figure}[t]
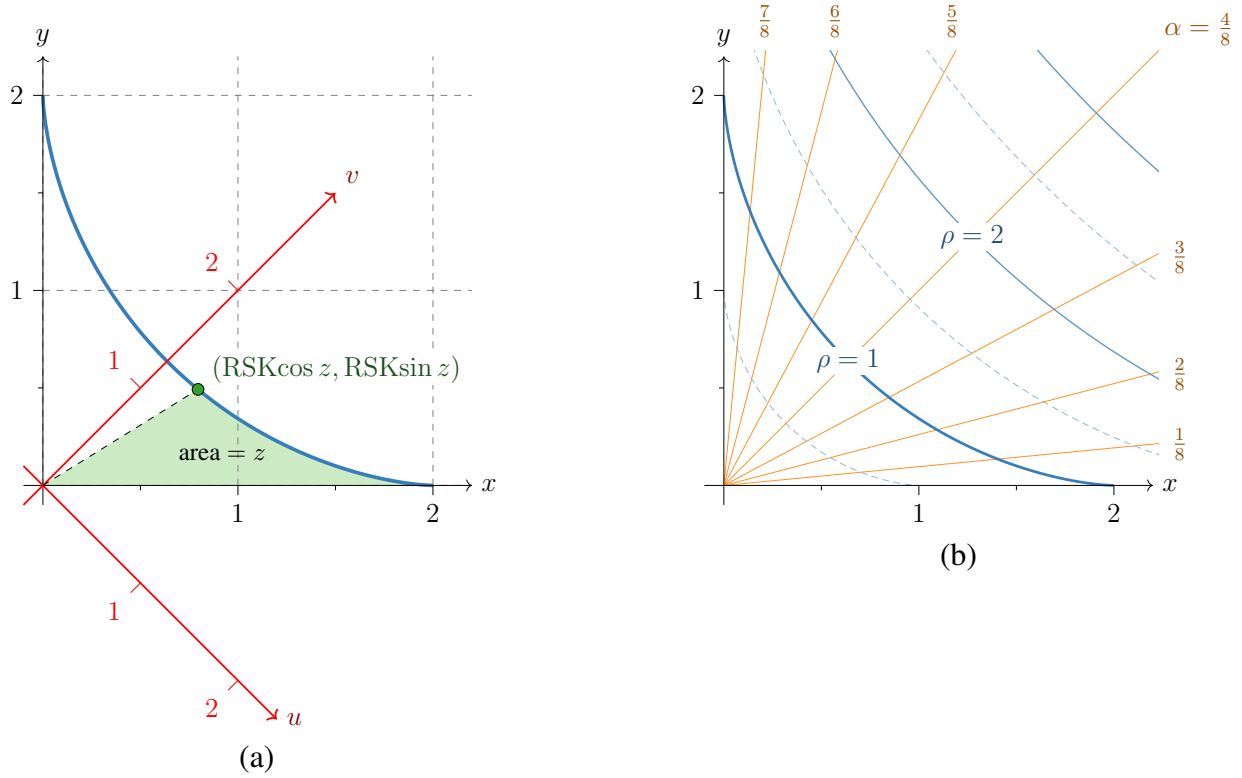

\begin{center}
\subfloat[]{%
	\label{fig:lsvk-a}%
	\scalebox{0.86}{\subfile{figures/FIGURE-curvilinear.tex}}%
}\hfill
\subfloat[]{%
	\label{fig:lsvk-b}%
	\scalebox{0.86}{\subfile{figures/FIGURE-rsk-polar.tex}}%
}
	
	\caption{\protect\subref{fig:lsvk-a} The
		Logan--Shepp--Vershik--Kerov (LSVK) limit shape
		(blue curve) in dual coordinate systems. Black axes: the French 
		coordinates $(x,y)$. Red axes: the Russian coordinates $(u,v)$ 
		with $u = x - y$, $v = x + y$. The curve is parametrized by 
		$(x,y)=(\operatorname{RSKcos} z, \operatorname{RSKsin} z)$ 
		for $z \in [0,1]$. The green shaded curvilinear triangle 
		has area equal to the parameter~$z$.
		\protect\subref{fig:lsvk-b} The RSK-polar coordinate system.}
	\label{fig:lsvk}
\end{center}
\end{figure}

\begin{lemma}
	\label{thm:curvilinear-triangle}
	For any $z \in [0,1]$, the area of the curvilinear 
	triangle $E(z)$ bounded by the $x$-axis, the ray from the origin 
	to $(\operatorname{RSKcos} z, \operatorname{RSKsin} z)$, 
	and the arc of the LSVK curve with parameters ranging in $[0,z]$,
	is equal to $z$.
\end{lemma}

\begin{proof}
Observe that in RSK-polar coordinates the region $E(z)$ maps
to the rectangle $\{ (\rho,\alpha) \,:\, 0\le \rho\le 1, \
0\le \alpha\le z  \}$. It follows that
\begin{align*}
\operatorname{area}(E(z)) &= \iint_{E(z)} \,dx \,dy =
\int_0^z \int_0^1 \left| \frac{\partial(x,y)}{\partial(\rho,\alpha)} \right|
d\rho \, d\alpha
=
\int_0^z \int_0^1 2\rho \, d\rho \, d\alpha = z.
\qedhere
\end{align*}
\end{proof}

\begin{remark}
\label{remark:ode-kerov-characterization}
The relation~\eqref{eq:omega-ode} is equivalent to the
statement that the \emph{radial distribution} of the curve
$\Omega(u)$, as defined by Kerov \cite[Sec.~4.4]{Kerov1996},
is the semicircle distribution. This fact was used by Kerov
to characterize $\Omega(\cdot)$ as the unique continual
Young diagram shape whose radial distribution coincides with
another natural distribution associated with the shape,
known as its transition measure.
\end{remark}

\begin{remark}
\label{remark:folklore-characterization}
The geometric characterization in \cref{thm:curvilinear-triangle} has circulated
as 
folklore, but we have not found a written proof in the 
literature. Generalizations have been discussed in the context 
of \emph{``geographic coordinate systems on continual Young diagrams,''} 
also without proof \cite[Remark~2.1]{MaslankaSniady2022}. 
\end{remark}

\begin{remark}
Our parametrization \eqref{eq:LSVK-parametrization} of the LSVK
curve runs in the opposite direction to the one implicit in
\cite{RomikSniady2015}: here the parameter increases from $z=0$ at
the point $(2,0)$ on the $x$-axis to $z=1$ at the point $(0,2)$ on
the $y$-axis, whereas in \cite{RomikSniady2015} the corresponding
parameter traverses the curve in the reverse order. Concretely, we
set $\scq_z = F^{-1}(1-z)$ in \eqref{eq:oneminusz-quantile} rather
than $F^{-1}(z)$; the two conventions are thus related by the
substitution $z \mapsto 1-z$, the reflection of the parameter
interval $[0,1]$ about its midpoint.
\end{remark}

\subsection{Asymptotic directions of planar sets}

\label{subsec:asymptotic-direction}

\begin{definition}[Asymptotic direction]
	\label{def:asymptotic-direction}
	Let $A \subset \N_0^2$ be an infinite set of boxes in the positive
	quarter-plane.
	We say that $A$ has an \emph{asymptotic direction} if the limit
	\[
		\asymdirect(A) := \lim_{n\to\infty} \alpha(x_n, y_n)
	\]
	exists for any enumeration $((x_n,y_n))_{n=1}^\infty$ of
	the elements of $A$ for which $x_n+y_n
	\xrightarrow[n\to\infty]{} \infty$. The limit value
	$\asymdirect(A) \in [0,1]$ is called the \emph{asymptotic
		RSK-angle of~$A$}, or, in a slight abuse of language, the
	asymptotic direction of~$A$.
\end{definition}

\section{Plancherel-normal and Plancherel-random IYTs}
\label{sec:proofs-plancherel-normal}

\subsection{Plancherel-normal tableaux}

We continue using
the notation $T_n$, $\origin_n$, and $\gamma_n$ from \cref{sec:proofs-general}.
Recall from \cref{sec:trajectories-trees-cusps} the jeu de taquin
trajectories $\jdtpath^T_{x,y}$ and their lazy parametrization; in this
notation the paths $\gamma_n$ are the trajectories from the origin in the
iterated tableaux $T_n$, that is, $\gamma_n = \jdtpath^{T_n}_{0,0}$.

We now extend the analysis of \cref{sec:proofs-general} to
the more restricted case of Plancherel-random IYTs. It will
be convenient to encapsulate the probabilistic nature of our
results by stating them in the generality of a class of IYTs
satisfying a certain minimal set of \emph{deterministic}
assumptions, which are known to hold for almost every IYT
chosen according to Plancherel measure. We call such IYTs
\emph{Plancherel-normal}.

\begin{definition}[Plancherel-normal tableau]
	\label{def:plancherel-normal}
	An IYT $T$ is called \emph{Plancherel-normal} if it satisfies the following
	conditions:
	\begin{enumerate}[label=\textup{(N\arabic*)}, ref=(N\arabic*)]
		\item \label{item:pn-asymdirect} Each jeu de taquin path
		$\gamma_n$ has an asymptotic direction
		\begin{equation} \label{eq:zn-asymdirect-gamman}
		      z_n := \asymdirect(\gamma_n).
		\end{equation}

		\item \label{item:pn-distinct} The values $z_1, z_2, z_3, \dots$ are
		distinct.

		\item \label{item:pn-dense} The set $\{z_n : n\geq 1\}$ is dense in $[0,1]$.
	\end{enumerate}
\end{definition}

The justification for this definition is contained in the following result.

\begin{proposition}[Plancherel-random tableaux are Plancherel-normal]
\label{prop:ae-Plancherel-normal}
Almost every IYT chosen at random according to Plancherel measure is
Plancherel-normal.
\end{proposition}

\begin{proof}
This is a consequence of the results of
\cite{RomikSniady2015} cited in the Introduction.
Specifically, let $T$ be a random IYT chosen according to
Plancherel measure. \cref{thm:romiksniady-main} implies that
$z_1,z_2,\ldots$ are i.i.d. random variables with the
uniform distribution $U(0,1)$. This gives properties
\labelcref{item:pn-distinct,item:pn-dense} in the definition
above. Moreover, the formula
\eqref{eq:inverse-rsk-firstcoord} is equivalent to the
statement that $z_1 = \asymdirect(\gamma_1)$. When combined
with the explicit formula \eqref{eq:inverse-rsk-explicit}
for the inverse map, this gives
\eqref{eq:zn-asymdirect-gamman}.
\end{proof}

\begin{assumption}[Standing assumption for this section]
	\label{ass:plancherel-normal}
	Throughout this section, $T$ is a fixed (deterministic)
	Plancherel-normal IYT in the sense of \cref{def:plancherel-normal}.
\end{assumption}

\noindent By \cref{prop:ae-Plancherel-normal}, any property that we prove
for $T$ under this assumption automatically holds with probability~$1$ for a
Plancherel-random IYT.

The next lemma explains that the standing assumption is
preserved under the jeu de taquin map; the proof is
immediate from the definitions (removing finitely many
points from a dense subset of $[0,1]$ leaves it dense), and
is omitted. 

\begin{lemma}[Plancherel-normality is preserved by jeu de taquin]
	\label{lem:pn-preserved-by-jdt}
	For every $k \geq 1$, the iterated tableau $T_k = \jdt^{k-1}(T)$ is again
	Plancherel-normal. Its jeu de taquin paths are
	$\gamma_k, \gamma_{k+1}, \ldots$, with respective asymptotic
	directions $z_k, z_{k+1}, \ldots$; in particular, the set of
	asymptotic directions of the jeu de taquin paths of $T_k$ is
	$\{z_n : n \geq k\}$.
\end{lemma}

\subsection{Confinement of trajectories to trees}

We establish that each path $\gamma_n$ eventually remains
within a single tree of~$T$.

\begin{lemma}[Eventual confinement to a single tree]
	\label{lem:trajectory_in_tree}
	For each $n \geq 1$, after removing a suitable finite
	initial segment from the path $\gamma_n$, the remaining tail becomes a jeu de
	taquin trajectory of the original tableau $T$. In particular, all but
	finitely many boxes of $\gamma_n$ lie within a single tree
	of~$T$.
\end{lemma}

\begin{proof}
	The path $\gamma_n$ is defined in the tableau $T_n$,
	which results from applying the jeu de taquin transformation
	$n-1$ times to $T$. Our strategy is to show that $\gamma_n$
	eventually enters a region where $T$ and $T_n$ differ only by
	relabeling: entries differ by the constant $n-1$, but their
	relative order is preserved.

	\medskip
	We start by observing a useful separation property: 
	since $T$ is Plancherel-normal, the asymptotic directions
	$z_1, \ldots, z_n$ are distinct.
	It follows that the paths must eventually separate:
	there is a finite initial segment of $\gamma_n$ whose removal
	leaves a tail $R$ such that no box of $R$ coincides with, or
	neighbors, any box visited by $\gamma_1, \ldots, \gamma_{n-1}$.
	Indeed, choose $\epsilon > 0$ smaller than a third of the
	minimal gap between the numbers $z_1, \ldots, z_n$. All but
	finitely many boxes of each $\gamma_k$ (for
	$1 \leq k \leq n$) have RSK-angle within $\epsilon$ of
	$z_k$; moreover, far from the origin, neighboring boxes
	have almost equal RSK-angles. Hence the tails of the
	paths, together with their neighboring boxes, are
	eventually confined to pairwise disjoint angular sectors.
	The finitely many exceptional boxes are avoided by removing
	from $\gamma_n$ a sufficiently long initial segment, since
	the boxes of a trajectory eventually leave any bounded
	region.
	
	\medskip
	This separation has a crucial consequence. The $n-1$
	applications of $\jdt$ that transform $T$ into $T_n$ perform
	sliding operations only in boxes visited by $\gamma_1, \ldots,
	\gamma_{n-1}$. Since the boxes of $R$ are separated from these,
	no sliding operation affects $R$ or its neighbors. The only
	change there is a global relabeling: each entry $k$ at a box of
	$R$ in $T$ becomes $k-(n-1)$ in $T_n$.
	
	This relabeling preserves all order relations between entries,
	hence the local comparison rules defining the jeu de taquin
	forest structure remain identical in both tableaux throughout
	$R$. Therefore, the tail $R$ of $\gamma_n$ follows the same
	directed edges in $T_n$ as it would in $T$. Being a directed
	path in $\jdtforest(T)$, the tail coincides with the
	trajectory $\jdtpath^T_{x_0,y_0}$ started at its first box
	$(x_0,y_0)$, and in particular lies entirely within a single
	tree of the jeu de taquin forest of $T$.
\end{proof}

\subsection{Asymptotic direction for the trees $\origin_n$}

The main result of this subsection is
\cref{prop:tree_uniform_limit}, which establishes that each tree
$\origin_n$ has a well-defined asymptotic direction equal to
$z_n$. We first prove this for $n=1$ in
\cref{lem:tree_uniform_limit}, then extend to all $n$ by
reducing to the $n=1$ case. We freely use the partial order $\preceq$,
the lazy parametrization, and the convex-corner properties of
\cref{sec:trajectories-trees-cusps}. We will also make use of the elementary
observation that $\preceq$
respects RSK-angular coordinates: if $(x_1,y_1) \preceq (x_2,y_2)$ then
$\alpha(x_1,y_1) \leq \alpha(x_2,y_2)$.

\subsubsection{Angular confinement of the tree at the origin}

We now establish that the asymptotic direction of a trajectory
is in fact a property of the entire tree containing that
trajectory.

\begin{lemma}[Asymptotic direction of the tree at the origin]
	\label{lem:tree_uniform_limit}
	The tree at the origin in $\jdtforest(T)$ has asymptotic direction $z_1$:
	\[
	\asymdirect(\origin_1) = z_1.
	\]
\end{lemma}

\begin{proof}
	We establish that for any $\epsilon > 0$, all but finitely
	many boxes $(x,y)$ in $\origin_1$ satisfy
	$|\alpha(x,y) - z_1| < \epsilon$. By symmetry, it suffices to
	prove the lower bound $\alpha(x,y) > z_1 - \epsilon$; the upper
	bound follows analogously.
	
	Our strategy is to construct comparison trajectories with
	known asymptotic directions and show that they bound
	trajectories within $\origin_1$ using the partial order 
	$\preceq$ that respects RSK-angular coordinates.
	
	\medskip
	\noindent\textbf{Step 1: Constructing a comparison trajectory.}
	If $z_1 = 0$ then the lower bound holds trivially, since
	$\alpha(x,y) \geq 0$ for every box; we may therefore assume
	$z_1 > 0$.
	Since $T$ is Plancherel-normal, the numbers
	$z_1, z_2, z_3, \ldots$ are dense in $[0,1]$. Given
	$\epsilon > 0$, choose an index $k$ such that
	$z_1 - \epsilon < z_k < z_1$.

	By \cref{lem:trajectory_in_tree}, after removing a finite
	initial segment from $\gamma_k$, the remaining tail becomes a jeu de taquin
	trajectory of $T$, lying within a single tree of $\jdtforest(T)$.
	Moreover, by possibly removing an additional finite prefix of this tail, we
	further restrict to a tail segment of $\gamma_k$ lying within a single tree of
	$\jdtforest(T)$, and with the additional property that all its boxes $(X,Y)$
	satisfy
	$\alpha(X,Y) > z_1 - \epsilon$; this is possible, since
	$\asymdirect(\gamma_k) = z_k > z_1 - \epsilon$ implies that all but
	finitely many boxes of $\gamma_k$ satisfy this inequality. Let $\jdtpath$
	denote
	this tail segment, denote its first box $(x_0, y_0)$, and denote
	$\tau_0 = T_{x_0, y_0}$. 
	Note that $\jdtpath$ coincides with
	the jeu de taquin trajectory $\jdtpath^T_{x_0,y_0}$ of the
	tableau $T$.

	\medskip
	\noindent\textbf{Step 2: Comparison setup.}
	Consider a box $(x,y) \in \origin_1$ with entry value
	$t_0 = T_{x,y} > \tau_0$. We establish the lower bound for all
	such boxes; since only finitely many boxes in $\origin_1$ have
	entries at most $\tau_0$, this suffices.
	
	Let $\jdtpath' = \jdtpath^T_{x,y}$ denote the jeu de taquin
	trajectory starting at $(x,y)$. 
	
	\medskip
	\noindent\textbf{Step 3: Trajectory ordering.}
	We claim that
	\begin{equation}
		\label{eq:ordered_by_prec}
		\jdtpath[t_0] \prec \jdtpath'[t_0].
	\end{equation}
	
	\smallskip
	\noindent\emph{Proof.}
	By \cref{lem:trajs-convex-corners}, 
	both
	$\jdtpath[t_0]$ and $\jdtpath'[t_0]$ are convex corners of
	$\lambda^{(t_0)} = \{(x',y') : T_{x',y'} \leq t_0\}$. Therefore by 
	\cref{lem:convex-corners}\ref{item:tc-2}, 
	they are
	comparable with respect to~$\preceq$: one is weakly northwest 
	or weakly southeast of the other.	
	
	Assume by contradiction that $\jdtpath[t_0] \succeq \jdtpath'[t_0]$.
	Applying \cref{lem:order-preservation} to
	the trajectories $\jdtpath'$ and $\jdtpath$ (in that order) at
	time~$t_0$, this ordering persists:
	$\jdtpath[t] \succeq \jdtpath'[t]$ for all $t \geq t_0$.
	By monotonicity of angular coordinates,
	$\alpha(\jdtpath[t]) \geq \alpha(\jdtpath'[t])$ for all $t\geq t_0$.
	
	\smallskip
	Now, since $\jdtpath$ is obtained by removing a finite prefix from
	$\gamma_k$, the asymptotic direction is preserved:
	$\asymdirect(\jdtpath) = \asymdirect(\gamma_k) = z_k$.

	On the other hand, the path $\jdtpath'$ lies entirely within $\origin_1$, as
	does
	$\gamma_1$. Since both paths lie in the same tree, they must
	eventually merge (\cref{lem:trajectories-merge}).
	Therefore, they share the same
	asymptotic direction:
	\[\asymdirect(\jdtpath')  =
	\asymdirect(\gamma_1)  = z_1.\]
	
	Now taking limits in $\alpha(\jdtpath[t]) \geq \alpha(\jdtpath'[t])$ yields
	\[
	z_k = \lim_{t \to \infty} \alpha(\jdtpath[t])
	\geq \lim_{t \to \infty} \alpha(\jdtpath'[t]) = z_1,
	\]
	contradicting $z_k < z_1$. Since we arrived at a contradiction,
	\eqref{eq:ordered_by_prec} must hold.
	
	\medskip
	\noindent\textbf{Step 4: Conclusion.}
	From \eqref{eq:ordered_by_prec}, we have
	$\jdtpath[t_0] \prec \jdtpath'[t_0] = (x,y)$. Since
	$\jdtpath[t_0]$ is one of the boxes of $\jdtpath$, all of
	which have RSK-angle exceeding $z_1 - \epsilon$, it follows
	that
	\[
	z_1 - \epsilon < \alpha(\jdtpath[t_0]) < \alpha(x,y).
	\]
	This holds for all $(x,y) \in \origin_1$ with
	$T_{x,y} > \tau_0$.
	As mentioned earlier, this suffices to finish the proof.
\end{proof}

\subsubsection{Angular confinement for all origin-containing trees $\origin_n$}

\Cref{lem:tree_uniform_limit} established angular confinement
for the tree $\origin_1$ containing the origin in $T$. We now
extend this result to all trees $\origin_n$.

\begin{proposition}[Asymptotic direction for all $\origin_n$]
	\label{prop:tree_uniform_limit}
	For each $n \geq 1$, the
	tree at the origin in $\jdtforest(T_n)$ has asymptotic
	direction $z_n$:
	\[
	\asymdirect(\origin_n)  = z_n.
	\]
\end{proposition}

\begin{proof}
	Set $T' := T_n = \jdt^{n-1}(T)$. $T'$ is Plancherel-normal
	by \cref{lem:pn-preserved-by-jdt}. Its jeu de taquin path
	from the origin is $\gamma_n = \jdtpath^{T'}_{0,0}$ by
	definition, and has asymptotic direction
	$\asymdirect(\gamma_n) = z_n$. The tree at the origin of
	$T'$ is precisely $\origin_n$. Therefore the result follows
	immediately by applying \cref{lem:tree_uniform_limit}
	to~$T'$. %
\end{proof}

\subsection{Applications of angular confinement}

Throughout this section, when we write $\tree \subseteq \tree'$,
$\tree \cap \tree'$, or $\tree \setminus \tree'$ for trees in
possibly different tableaux, we mean set-theoretic operations on
the underlying sets of boxes (as elements of $\N_0^2$), ignoring
the graph structure and tableau entries. This allows us to
compare spatial positions of boxes across different stages of the
jeu de taquin transformation.

\subsubsection{Trees at the origin have finite intersections}

\begin{corollary}[Finite intersection property for $\origin_n$]
	\label{thm:trees-almost-disjoint}
	For all positive integers
	$i \neq j$, the intersection $\origin_i \cap \origin_j$ is
	finite.
\end{corollary}

\begin{proof}
	By the Plancherel-normality assumption, $z_i \neq z_j$. Choose
	$\epsilon < |z_i - z_j|/3$. By \cref{prop:tree_uniform_limit},
	all but finitely many boxes in $\origin_i$ have angular
	coordinate in $(z_i - \epsilon, z_i + \epsilon)$, and
	similarly for $\origin_j$. Since
	\[
	|z_i - z_j| > 3\epsilon,
	\]
	the intervals $(z_i - \epsilon, z_i + \epsilon)$ and
	$(z_j - \epsilon, z_j + \epsilon)$ are disjoint.
	
	It follows that for any box $(x,y) \in \origin_i \cap
	\origin_j$, either $\alpha(x,y) \notin (z_i-\epsilon,
	z_i+\epsilon)$ or $\alpha(x,y) \notin (z_j-\epsilon,
	z_j+\epsilon)$. The sets of boxes satisfying each of these
	conditions are finite. Therefore, $\origin_i \cap
	\origin_j$ is finite.
\end{proof}

\subsubsection{Finite growth of peripheral trees}

In \cref{thm:jdt-forest-evolution-weak}, we established that for
any peripheral tree $\tree$ in $T$, its image $\iota(\tree)$
under the jeu de taquin transformation contains $\tree$ as a
subtree. The following result shows that for Plancherel-normal
tableaux, this growth is bounded: each peripheral tree acquires
only finitely many new boxes.

\begin{corollary}[Peripheral trees grow by finitely many boxes]
	\label{thm:finite-growth}
	For each peripheral tree
	$\tree$ in $T$, the set difference
	$\iota(\tree) \setminus \tree$ is finite.
\end{corollary}

\begin{proof}
	By \cref{prop:absorbed-boxes}, the boxes absorbed during the
	first transformation satisfy
	$\iota(\tree) \setminus \tree \subseteq \origin_1 \cap
	\origin_n$ for some $n \geq 2$. By
	\cref{thm:trees-almost-disjoint}, each intersection
	$\origin_1 \cap \origin_n$ is finite. Therefore,
	$\iota(\tree) \setminus \tree$ is finite.
\end{proof}

\subsubsection{Maximal extensions differ by finitely many boxes}

\begin{corollary}[Finite difference from maximal extension]
	\label{thm:max-extension}
	For each tree $\tree$ in the
	jeu de taquin forest of $T$, the difference
	$\tree_{\max} \setminus \tree$ between the tree and its
	maximal extension is finite.
\end{corollary}

\begin{proof}
	By repeated application of the bijection $\iota$ from
	\eqref{eq:def-iota}, we trace the evolution of $\tree$
	through successive tableaux until it reaches the origin after
	finitely many steps. At each step, by \cref{thm:finite-growth},
	the growth is finite. Since the total number of steps is
	finite, the cumulative growth $\tree_{\max} \setminus \tree$
	is a finite union of finite sets, hence finite.
\end{proof}

\subsubsection{Surjectivity of tree evolution}

We now show that, for Plancherel-normal tableaux, the jeu de taquin
transformation produces
no new trees: every tree in $\jdt(T)$ arises as the image
of a peripheral tree from $T$.

\begin{proposition}[Surjectivity of tree evolution]
	\label{prop:no-residue}
	The injective map $\iota$
	from~\eqref{eq:def-iota} is a bijection. That
	is, every tree in $T_2 = \jdt(T)$ is the image under $\iota$
	of some peripheral tree in $T$.
\end{proposition}

\begin{proof}
	By \cref{thm:trees-almost-disjoint}, for all $n \geq 2$, the
	intersection $\origin_1 \cap \origin_n$ is finite.
	
	Assume by contradiction that there exists a tree $\tree$ in
	$T_2$ that is not in the image of $\iota$. Since $\iota$ maps
	all peripheral trees of $T$ injectively to trees in $T_2$, and
	$\tree$ is not in the image, the boxes of $\tree$ cannot have
	come from any peripheral tree. By
	\cref{thm:jdt-forest-evolution-weak}, the only other source of
	boxes during the transformation is the tree at the origin
	$\origin_1$ in $T$. Therefore,
	\[
	\tree \subseteq \origin_1.
	\]
	
	On the other hand, by \cref{rem:lifetime-after-jdt}, the tree
	$\tree$ in $T_2$ satisfies
	$\tree \subseteq \tree_{\max} = \origin_n$ for some
	$n \geq 2$. Combining these inclusions,
	\[
	\tree \subseteq \origin_1 \cap \origin_n.
	\]
	
	Since $\origin_1 \cap \origin_n$ is finite, we conclude that
	$\tree$ contains only finitely many boxes. However, as
	observed in \cref{sec:intro}, every tree in a jeu de taquin
	forest is infinite. This contradiction shows that no
	such tree $\tree$ can exist. Therefore $\iota$ is surjective, hence a
	bijection.
\end{proof}

\subsection{Lifetime is a bijection}

\begin{proposition}
	\label{prop:lifetime-bijection}
	The lifetime map
	\[
	\tree \mapsto n(\tree)
	\]
	defines a bijection between the set of all trees in $\jdtforest(T)$ and
	the positive integers~$\N$.
\end{proposition}

\begin{proof}
	We establish that the lifetime map is a bijection by proving
	injectivity and surjectivity separately.
	
	\medskip
	\noindent\textbf{Injectivity.}
	Suppose two trees $\tree, \tree'$ in $T$ have the same
	lifetime $n$. By \cref{thm:jdt-forest-evolution-weak}, the map
	$\iota$ induces an injection between trees at each step. After
	$n-1$ iterations of the jeu de taquin transformation, both
	trees reach the origin: $\iota^{n-1}(\tree)$ and
	$\iota^{n-1}(\tree')$ are both trees at the origin of tableau
	$T_n$. Since there is a unique tree at the origin of any
	tableau, we have
	$\iota^{n-1}(\tree) = \iota^{n-1}(\tree')$. Since
	$\iota^{n-1}$ is an injection (being the composition of
	injections), we conclude $\tree = \tree'$.
	
	\medskip
	\noindent\textbf{Surjectivity.}
	Let $n$ be a positive integer. Consider the tree $\origin_n$
	at the origin in $T_n$. By \cref{prop:no-residue}, 
	the map $\iota$ is a bijection at each step; here we use
	\cref{lem:pn-preserved-by-jdt}, by which each of the
	tableaux $T_1, \ldots, T_{n-1}$ is again Plancherel-normal,
	so that \cref{prop:no-residue} applies to it.
	Tracing $\origin_n$ backward through the inverse maps, we
	obtain a unique tree in $T$ with lifetime equal to $n$.
	
	Therefore, the lifetime map is a bijection.
\end{proof}

\subsection{Asymptotic directions for all jeu de taquin forest trees}

\begin{proposition}[Asymptotic directions for all jeu de taquin forest trees]
	\label{prop:pn-tree-properties}
	The trees of $\jdtforest(T)$ satisfy:
	\begin{enumerate}[label=(T\arabic*)]
		\item \label{item:tree-has-direction}
		Each tree $\tree$ in $\jdtforest(T)$ has an asymptotic direction
		\[ z_{\tree} := \asymdirect(\tree)  \in [0,1]. \]
		
		\item \label{item:tree-directions-distinct}
		Distinct trees have distinct asymptotic directions.
		
		\item \label{item:tree-directions-equal-zn}
		The set of asymptotic directions of the trees coincides with
		the set of inverse-RSK values:
		\begin{equation}
			\label{eq:ztreeset-equals-znset}
			\{z_{\tree} : \tree \text{ a tree in } \jdtforest(T)\}
			= \{z_n : n \geq 1\}.
		\end{equation}

		\item \label{item:tree-directions-dense}
		The set
		$\{z_{\tree} : \tree \text{ a tree in } \jdtforest(T)\}$ is dense in
		$[0,1]$.
	\end{enumerate}
\end{proposition}

The key to establishing these properties is a bijective
correspondence between trees and trajectories, which we now develop.

\subsubsection{Proof via trajectory correspondence}

The following result establishes that each tree
corresponds to a unique jeu de taquin trajectory.

\begin{proposition}[Trajectories correspond to trees]
	\label{thm:pn-trajectory-tree-bijection}
	For each tree $\tree$ in $\jdtforest(T)$, its
	lifetime $n(\tree)$ is the unique positive integer $k$ such
	that the path $\gamma_k$, after removing a finite
	initial segment, lies entirely within $\tree$.
	
	Moreover, the asymptotic directions exist and coincide:
	\begin{equation} \label{eq:asymdirections-coincide}
	\asymdirect(\tree)  = \asymdirect(\gamma_{n(\tree)}) =
	z_{n(\tree)}
	\end{equation}
	for all trees $\tree$ in $\jdtforest(T)$.
\end{proposition}

\begin{proof}
	Let $\tree$ be any tree in $\jdtforest(T)$, and let $k = n(\tree)$ denote
	its lifetime. We establish the claims about $k$ and its uniqueness
	in three steps.
	
	\medskip
	\noindent\textbf{Step 1: The path $\gamma_k$ lies in
		$\tree$ after a finite prefix.}
	The maximal extension $\tree_{\max} = \origin_k$ is the tree
	at the origin in $\jdtforest(T_k)$, which contains the jeu de taquin
	path $\gamma_k = \jdtpath^{T_k}_{0,0}$. By
	\cref{thm:max-extension}, the trees $\tree$ and $\tree_{\max}$
	differ by finitely many boxes. Therefore, after removing a
	finite initial segment, the path $\gamma_k$ is contained
	in $\tree$.
	
	\medskip
	\noindent\textbf{Step 2: The asymptotic directions coincide.}
	By \cref{prop:tree_uniform_limit}, 
	the asymptotic direction $\asymdirect(\origin_k)$
	exists and equals $z_k = \asymdirect(\gamma_k)$. 
	Since $\tree$ is a subset of $\tree_{\max}=\origin_k$,
	the asymptotic direction $\asymdirect(\tree)$ also exists
	and equals the same number $z_k = \asymdirect(\gamma_k)$; 
	this establishes \eqref{eq:asymdirections-coincide}.
	
	\medskip
	\noindent\textbf{Step 3: Uniqueness of $n(\tree)$.}
	Suppose that for some index $\ell$, the path
	$\gamma_\ell$ with a finite initial segment removed is
	contained in $\tree$. Then by the same reasoning as Step~2,
	\[
	\asymdirect(\tree) = 
	\asymdirect(\gamma_\ell)  =
	z_\ell.
	\]
	By Step~2, we also have
	$\asymdirect(\tree) = z_{n(\tree)}$. Therefore,
	$z_\ell = z_{n(\tree)}$. Since $T$ is Plancherel-normal, the values
	$z_1, z_2, z_3, \ldots$ are distinct, hence
	$\ell = n(\tree)$, establishing uniqueness.
\end{proof}

\subsubsection{Proof of \cref{prop:pn-tree-properties}}
We can now complete the proof of
\cref{prop:pn-tree-properties}.

\begin{proof}[Proof of \cref{prop:pn-tree-properties}]
	\emph{Property \ref{item:tree-has-direction}:} This follows
	immediately from \cref{thm:pn-trajectory-tree-bijection},
	which establishes that $\asymdirect(\tree) = z_{n(\tree)}$ exists
	for each tree $\tree$.
	
	\medskip
	\emph{Property \ref{item:tree-directions-distinct}:} If
	$\tree \neq \tree'$ are distinct trees, then by injectivity of
	the lifetime map (\cref{prop:lifetime-bijection}), we have
	$n(\tree) \neq n(\tree')$. By
	\cref{thm:pn-trajectory-tree-bijection},
	$z_{\tree} = z_{n(\tree)}$ and
	$z_{\tree'} = z_{n(\tree')}$. Since $T$ is Plancherel-normal, the values
	$z_1, z_2, z_3, \ldots$ are distinct, hence
	$z_{\tree} \neq z_{\tree'}$.
	
	\medskip
	\emph{Property \ref{item:tree-directions-equal-zn}:} By
	\cref{thm:pn-trajectory-tree-bijection}, $z_{\tree} = z_{n(\tree)}$
	for each tree $\tree$, so the left-hand side of
	\eqref{eq:ztreeset-equals-znset} is contained in the right-hand
	side. Conversely, by surjectivity of the lifetime map
	(\cref{prop:lifetime-bijection}), for each $n \geq 1$ there
	exists a tree $\tree$ with $n(\tree) = n$, whence $z_{\tree} = z_n$
	by the same correspondence. This proves the equality of sets
	\eqref{eq:ztreeset-equals-znset}.

	\medskip
	\emph{Property \ref{item:tree-directions-dense}:} Since $T$ is
	Plancherel-normal, the set $\{z_n : n \geq 1\}$ is dense in
	$[0,1]$. By \eqref{eq:ztreeset-equals-znset}, the same holds for
	the set of asymptotic directions of the trees, completing the
	proof.
\end{proof}

\subsection{Proofs of \cref{thm:main-result,thm:complete-inverse}}

\label{subsec:proof-thm-main-result}

By \cref{prop:ae-Plancherel-normal}, a Plancherel-distributed random
IYT $T$ is almost surely Plancherel-normal, so we may apply
\cref{prop:pn-tree-properties}. Its
property~\ref{item:tree-has-direction} yields the first statement, and
its property~\ref{item:tree-directions-equal-zn} yields the equality of
sets
\[
\left\{ \asymdirect(\tree) \,:\, \tree \textrm{ is a tree in }\jdtforest(T)
\right\}
= \{ z_1, z_2, \ldots \},
\]
which is the second statement. Finally, its
property~\ref{item:tree-directions-distinct} yields the third
statement.

\Cref{thm:complete-inverse} follows in the same way:
\cref{prop:lifetime-bijection} shows that the lifetime map is a
bijection onto $\N$, and \cref{thm:pn-trajectory-tree-bijection}
identifies the asymptotic direction of the tree of lifetime $n$
as~$z_n$.

\subsection{Trajectories repeatedly approach tree boundaries}

A box $(x,y) \in A$ is called a \emph{boundary box} of a set
$A \subseteq \N_0^2$ if $(x+1,y-1)$ or $(x-1,y+1)$ lies outside of $A$.

\begin{theorem}[Trajectories repeatedly approach tree boundaries]
	\label{thm:pn-infinite-boundary-boxes}
	For any starting position
	$(x,y) \in \N_0^2$, the jeu de taquin trajectory
	$\jdtpath^T_{x,y}$ passes through infinitely many boundary
	boxes of the tree $\tree$ of $\jdtforest(T)$ that contains it.
\end{theorem}

\begin{proof}
	\noindent\textbf{Step 1: Reduction to paths $\gamma_n$.}
	The starting position $(x,y)$ belongs to some tree $\tree$ in
	the jeu de taquin forest of $T$. By the lifetime bijection
	(\cref{prop:lifetime-bijection}), we have
	$\tree_{\max} = \origin_n$ for some $n \geq 1$.
	
	By \cref{thm:max-extension}, the trees $\tree$ and
	$\tree_{\max}$ differ by finitely many boxes. Consequently,
	their sets of boundary boxes differ by at most finitely many
	elements. Therefore, it suffices to prove that
	$\tree_{\max} = \origin_n$ has infinitely many boundary boxes
	visited by the trajectory $\jdtpath^{T}_{x,y}$.
	
	By \cref{lem:trajectory_in_tree}, after a finite initial
	segment the path $\gamma_n$ follows the forest edges of
	$T$ and lies within a single tree of $\jdtforest(T)$; by
	\cref{thm:pn-trajectory-tree-bijection}, this tree is $\tree$,
	the tree of lifetime $n$. Two trajectories of $\jdtforest(T)$
	lying in the same tree eventually merge, so $\jdtpath^{T}_{x,y}$
	and $\gamma_n$ visit the same boxes after removing finite
	initial segments. Therefore, it suffices to prove that for each
	$n \geq 1$, the path $\gamma_n$ passes through infinitely
	many boundary boxes in $\origin_n$.
	
	\medskip
	\noindent\textbf{Step 2: Reduction to the case $n=1$.}
	By \cref{lem:pn-preserved-by-jdt}, the tableau
	$T_n = \jdt^{n-1}(T)$ is itself Plancherel-normal, with sliding
	paths $\gamma_{n+k}$ of asymptotic directions $z_{n+k}$. The
	argument is therefore identical for all $n$, and it suffices to
	prove the result for $n = 1$. That is, we prove
	that $\gamma_1 = \jdtpath^{T}_{0,0}$ passes through infinitely
	many boundary boxes in $\origin_1$.
	
	\medskip
	\noindent\textbf{Step 3: Existence of infinitely many flip
		boxes.}
	We compare the outgoing connectivity of boxes in $\origin_1$
	in two different jeu de taquin forests. Call a box
	$(x,y) \in \origin_1$ a \emph{flip box} if its outgoing edge
	in the forest of $T_2 = \jdt(T)$ points to a box that was
	peripheral in $T$ (i.e., not in $\origin_1$).
	
	We claim that there exist infinitely many flip boxes. Suppose
	for contradiction that there are only finitely many flip
	boxes. Let $M$ denote the maximum entry $(T_2)_{x,y}$ of the
	tableau $T_2$ taken over the finitely many flip boxes $(x,y)$.

	Since $\origin_1$ is infinite, there exists a box
	$(x,y) \in \origin_1$ with entry $(T_2)_{x,y} > M$. Follow the
	jeu de taquin trajectory $\jdtpath^{T_2}_{x,y}$ in the tableau $T_2$
	from $(x,y)$; the entries of $T_2$ increase along the trajectory, so
	every box of the trajectory has $T_2$-entry exceeding $M$ and is
	therefore not a flip box. Starting from $(x,y) \in \origin_1$,
	an induction now shows that the trajectory remains within $\origin_1$:
	any box of the trajectory lying in $\origin_1$ is a non-flip box, so
	its outgoing edge in the forest of $T_2$ points within
	$\origin_1$, placing the next box of the trajectory in $\origin_1$ as
	well. Therefore, this trajectory remains within $\origin_1$
	indefinitely.

	This is impossible. The trajectory $\jdtpath^{T_2}_{x,y}$ is an
	infinite jeu de taquin trajectory in $T_2 = \jdt(T)$, so all of its
	boxes belong to a single tree $\tree'$ of $\jdtforest(T_2)$. By
	\cref{prop:no-residue}, $\tree' = \iota(\tree)$ for some
	peripheral tree $\tree$ of $T$. Since $\tree$ is peripheral, it
	is disjoint from the tree $\origin_1$ at the origin; as the trajectory
	lies in $\origin_1$, it avoids $\tree$ entirely, so that
	\[
	\jdtpath^{T_2}_{x,y} \subseteq \tree' \setminus \tree
	= \iota(\tree) \setminus \tree.
	\]
	By \cref{prop:absorbed-boxes}, the set
	$\iota(\tree) \setminus \tree$ is contained in
	$\origin_1 \cap \origin_n$ for some $n \geq 2$, which is finite
	by \cref{thm:trees-almost-disjoint}. Thus the infinite trajectory
	$\jdtpath^{T_2}_{x,y}$ would be contained in a finite set, a
	contradiction.

	Therefore, there are infinitely many flip boxes in $\origin_1$.
	
	\medskip
	\noindent\textbf{Step 4: Flip boxes imply boundary box
		visits.}
	We show that each flip box in $\origin_1$ has a neighboring
	boundary box that $\gamma_1$ visits. Consider a flip box
	$(x',y') \in \origin_1$. By definition, its outgoing edge in
	the forest of $T_2$ points to a box outside $\origin_1$; in the
	forest of $T$, by contrast, its outgoing edge stays within its
	own tree $\origin_1$.
	
	The outgoing edge from $(x',y')$ is determined by comparing
	the entries of its two neighbors: north $(x', y'+1)$ and east
	$(x'+1, y')$. The edge points to whichever neighbor has the
	smaller entry. For the outgoing edge to change direction, one
	of the two neighbors must have had its entry modified by the
	jeu de taquin sliding operation. This neighbor must therefore
	lie on the jeu de taquin path $\gamma_1$.
	
	Without loss of generality, suppose the eastern neighbor
	$(x'+1, y')$ lies on $\gamma_1$. (The case where the northern
	neighbor lies on $\gamma_1$ is symmetric, being its mirror image
	under the NE--SW diagonal.) Set $(x,y) = (x'+1, y')$; we claim
	that $(x,y)$ is a boundary box of $\origin_1$ visited by
	$\gamma_1$. The geometric configuration is illustrated in
	\cref{fig:boundary-boxes}, where the flip box is marked with a
	red square and the boundary box with $\times$.

	Since $(x,y)$ lies on $\gamma_1$, it belongs to $\origin_1$. It
	remains to show that its northwestern neighbor
	$(x-1, y+1) = (x', y'+1)$ lies outside $\origin_1$. After the
	slide, the outgoing edge of the flip box $(x',y')$ points
	outside $\origin_1$, to whichever of its northern neighbor
	$(x', y'+1)$ and eastern neighbor $(x,y)$ has the smaller entry.
	As $(x,y) \in \origin_1$, this edge cannot point east; hence it
	points north, to $(x', y'+1)$, which therefore lies outside
	$\origin_1$. (This is the transition from an eastward to a
	northward edge depicted in \cref{fig:boundary-boxes}.)
	Consequently $(x,y) \in \origin_1$ while
	$(x-1, y+1) \notin \origin_1$, so $(x,y)$ is a boundary box of
	$\origin_1$.

	Each flip box $(x',y')$ thus determines an adjacent boundary
	box of $\origin_1$ lying on $\gamma_1$ (namely $(x'+1,y')$ or
	$(x',y'+1)$). A given boundary box arises in this way from at
	most two flip boxes---its western and its southern
	neighbor---so this assignment is at most two-to-one. Since
	there are infinitely many flip boxes, we conclude that
	$\gamma_1$ passes through infinitely many boundary boxes of
	$\origin_1$.
\end{proof}

\begin{figure}[t]
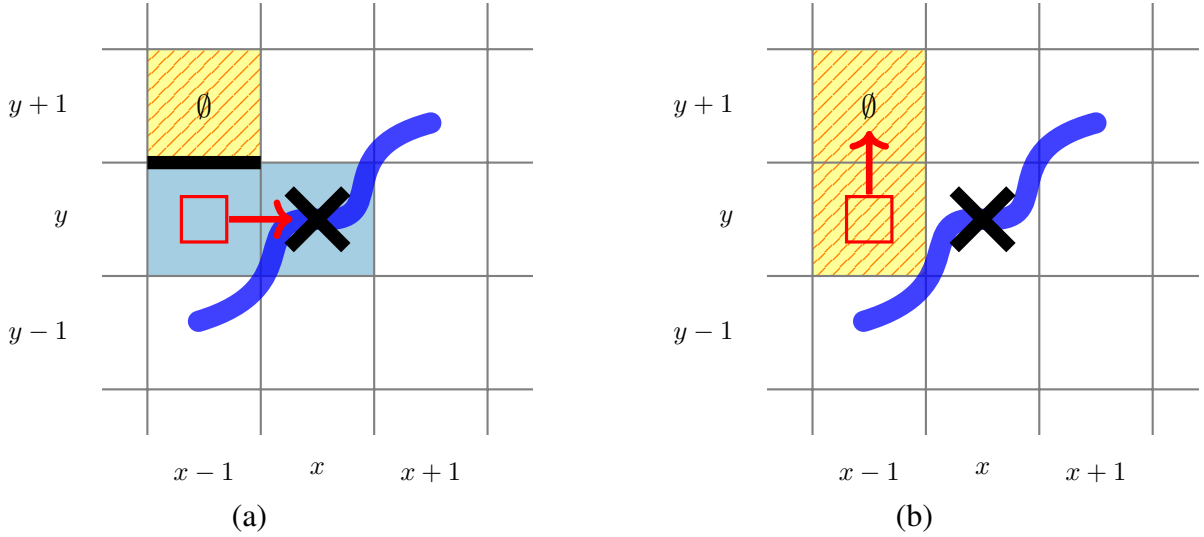

	\centering
	\subfloat[]{%
		\label{fig:boundary-boxes-a}%
		\subfile{figures/FIGURE-boundary_box-before.tex}%
	}\hfill
	\subfloat[]{%
		\label{fig:boundary-boxes-b}%
		\subfile{figures/FIGURE-boundary_box-after.tex}%
	}

	\caption{A jeu de taquin slide at a boundary box $(x,y)$ (large
		$\times$) and its \emph{flip box} \mbox{$(x-1,y)$} (red
		square), illustrating \cref{thm:pn-infinite-boundary-boxes}.
		The flip box's outgoing edge is red because it changes
		direction across the slide: it points
		\protect\subref{fig:boundary-boxes-a} east, into the origin
		tree $\origin_1$ (light blue), before the slide, and
		\protect\subref{fig:boundary-boxes-b} north afterwards, out of
		$\origin_1$ towards $(x-1,y+1)$ (striped yellow, a neighbouring
		peripheral tree marked $\emptyset$): the origin tree dissolves
		here and $(x,y)$ is left unpainted, its membership
		undetermined. The blue snake is the jeu de taquin path through
		$(x,y)$, drawn vaguely as its entry and exit sides are
		undetermined; the complementary north${}\to{}$east case is the
		mirror image under the NE--SW diagonal.}
	
	\label{fig:boundary-boxes}
\end{figure}

\section{Open problems}
\label{sec:final-remarks}

In this final section we collect a number of numerical observations and open
problems
suggested by the
jeu de taquin forest: the point process of tree cusps and its conjectural
coalescing-flow scaling limit, the level repulsion of the asymptotic
directions, the hierarchical tree-of-trees structure, a random permutation
of $\N$ comparing two orderings of the trees, and the clustering of
boundary-box visits along a trajectory.

\subsection{The point process of tree cusps}

By \cref{thm:cusp-in-tree}, every tree $\tree$ of the jeu de taquin
forest carries a well-defined \emph{cusp} $\cusp(\tree)$, the
$\trianglelefteq$-minimal box of $\tree$. For a Plancherel-random IYT this
produces a random point process $\{\cusp(\tree)\}$, indexed by the trees of
the forest, whose statistics --- intensity, correlations, and scaling limit
--- it would be interesting to describe (\cref{fig:cusp-forest}). We collect
here a few numerical observations, which we leave as open problems.

\begin{figure}[t]
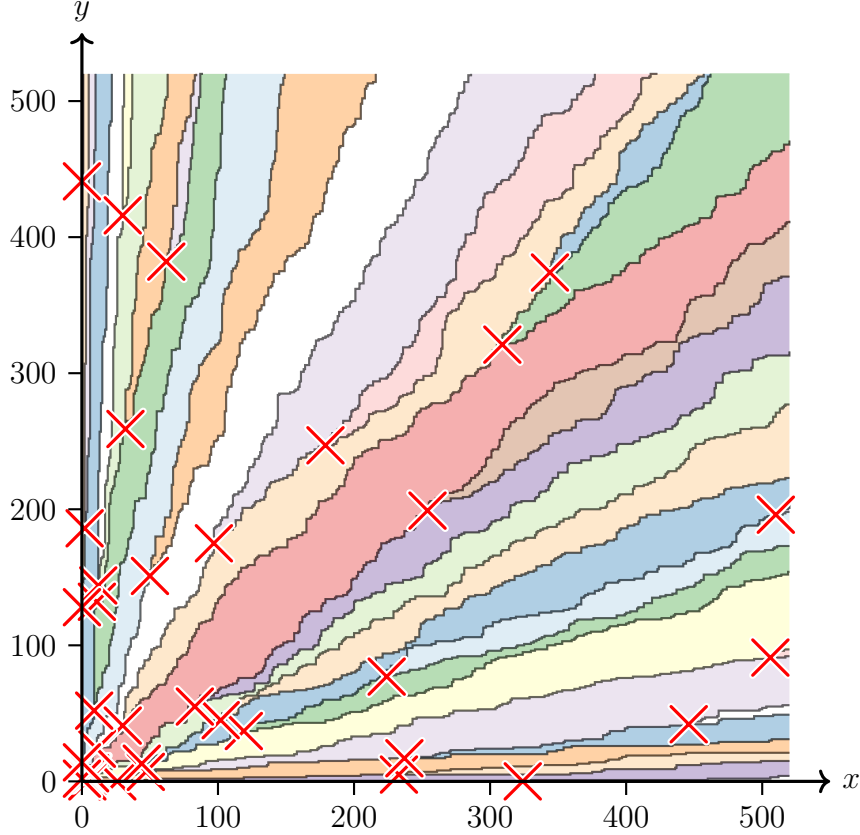

	\subfile{figures/FIGURE-jdt_competition_LSVK_ultra2-cusps.tex}
	\caption{The point process of tree cusps for a large
		Plancherel-random tableau, the same realization as in
		\cref{fig:jdt-forest}. Each tree occupies a thin
		wedge-shaped region extending toward its asymptotic
		direction, and the cusp of a tree --- its
		$\trianglelefteq$-minimal box, equivalently its
		minimum-entry box (\cref{cor:cusp-min-entry}) --- is the
		corner where its region tapers to a point on the side of the
		origin. The wedge colours are subdued and a bold red cross
		marks each cusp, so that the cusps stand out; they grow
		sparser away from the origin. The wedges visibly crowd
		together near the two coordinate axes, reflecting the
		concentration of asymptotic directions discussed in the
		text.}
	\label{fig:cusp-forest}
\end{figure}

We use the RSK-polar coordinates $(\rho, \alpha)$ of \cref{sec:LSVK-curve}.

First, the cusps grow sparser away from the origin. We call the
smallest entry $\min \tree$ occurring in a tree $\tree$ --- by
\cref{cor:cusp-min-entry}, the entry at its cusp box --- the
\emph{bifurcation time} of the tree. In all Monte Carlo
experiments of this section, a \emph{realization} is a
Plancherel-random IYT restricted to its first $n = 10^6$ boxes,
that is, to the Young diagram $\lambda^{(n)}$ of \cref{sec:intro};
such a realization exhibits precisely the cusps with bifurcation
time at most $n$, so the truncation size $n$ acts as a threshold
on the bifurcation time. The number of cusps with
bifurcation time at most $t$ appears to grow like $K\,t^{1/4}$ as
$t \to \infty$, for some absolute constant $K \approx 1.87$ whose
value we do not determine (\cref{fig:cusp-growth}).

%	EASTER EGG: the conjectural value of K is approximately
%	K = 1.86837 22169 05414 14351 97445 73373 77269 14483
%	32194 97193 28399 88660 67471 83563 41340 02490
%	85495 81062 74568 27215

\begin{figure}[t]
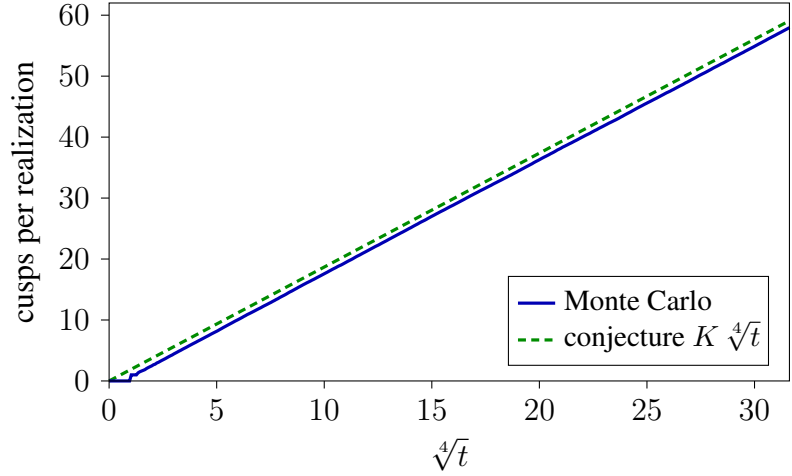

	\centering
	% Generated by cutoff_sweep_tikz_export.py; see the header of
	% figures/cutoff-sweep.tex for the data source and regeneration steps.
	\subfile{figures/cutoff-sweep.tex}
	\caption{The number of cusps with bifurcation time at most $t$, per
		realization, against the radial coordinate $\sqrt[4]{t}$ (solid;
		Plancherel realizations with $n = 10^6$ boxes), compared with a straight line
		of slope $K \approx 1.87$ (dashed). The count grows linearly in
		$\sqrt[4]{t}$, the data lying a near-constant offset below the line
		(a finite-size correction).}
	\label{fig:cusp-growth}
\end{figure}

Second, the asymptotic directions concentrate near the two coordinate axes
$z = 0$ and $z = 1$, with a limiting law that appears singular there
(\cref{fig:cusp-direction-cdf}); although \cref{prop:pn-tree-properties}
shows the set of asymptotic directions to be dense in $[0,1]$, we do not
determine this law. This is not in tension with the input sequence
$z_1,z_2,\ldots$ being i.i.d.\ uniform on $[0,1]$: by
\cref{thm:main-result} every tree's asymptotic direction is one of the
values $z_n$ and every $z_n$ is realized. But \cref{fig:cusp-direction-cdf}
records only the cusps inside a bounded window, ordered by bifurcation time
--- a geometrically selected, and hence biased, subsample rather than the
full i.i.d.\ sequence.

\begin{figure}[t]
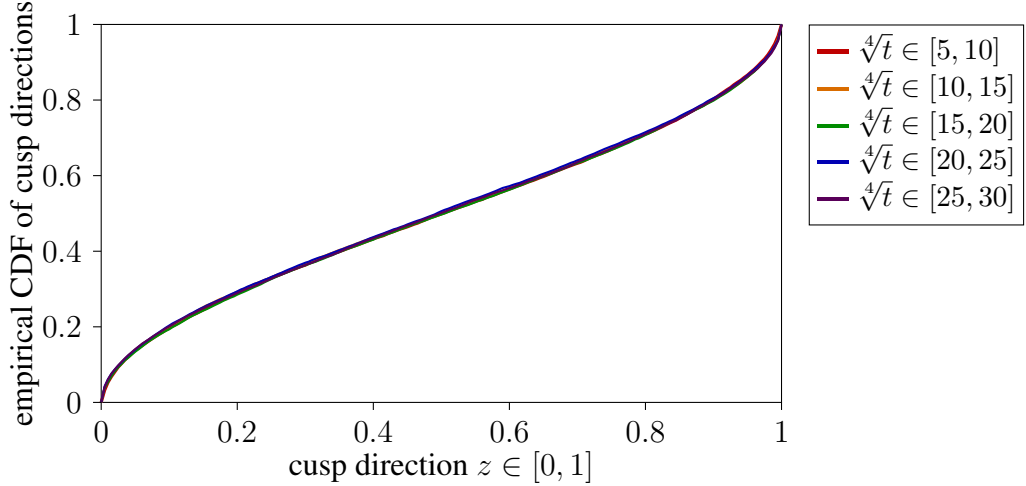

	\centering
	% Generated by cusp_direction_cdf_export.py; see the header of
	% figures/cusp-direction-cdf.tex for the data source and regeneration.
	\subfile{figures/cusp-direction-cdf.tex}
	\caption{Empirical cumulative distribution of the asymptotic directions
		$z \in [0,1]$ of the cusps, in radial shells $\sqrt[4]{t} \in [5,30]$
		(Plancherel realizations with $n = 10^6$ boxes). The curves rise
		steeply near
		the two axes $z = 0, 1$ and flatten in between --- the signature of a
		limiting direction law that concentrates at, and is singular at, the
		axes. Although the input values $z_n$ are i.i.d.\ uniform, this
		finite, geometrically selected sample of cusps is biased relative to
		the whole sequence.}
	\label{fig:cusp-direction-cdf}
\end{figure}

The picture as a whole --- the cusps as the points where the competition
interfaces, drawn inward from spatial infinity, merge --- is reminiscent of a
system of coalescing Brownian motions, in the same universality class as the
exactly solvable P\'olya web~\cite{BSTU2026polyaWeb}. Making this precise,
and with it the exact growth constant $K$ and the limiting angular law, is an
open problem.

\subsection{Level repulsion of the asymptotic directions}
\label{subsec:level-repulsion}

Beyond the global angular law, the asymptotic directions appear to
\emph{repel}. Fixing a threshold on the bifurcation time, forming
the spacings between consecutive directions, normalizing them by
the empirical mean density of the directions (estimated from all
realizations, to compensate for its non-uniformity), and
restricting to a bulk window of directions away from the two
axes, one finds in
simulations a limiting spacing law whose density vanishes ---
linearly --- as the spacing tends to $0$, conspicuously unlike the
exponential law $e^{-s}$ of independent (Poisson) points
(\cref{fig:rayleigh-spacing}). Identifying this limiting spacing
law is another open problem.

\begin{figure}[t]
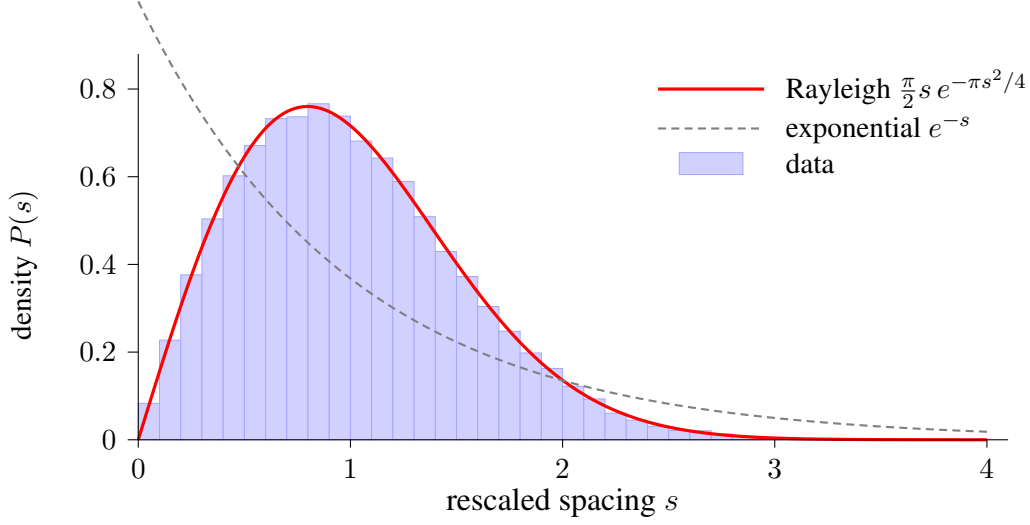

	\centering
	% Generated by rayleigh_spacing_tikz_export.py; see the header of
	% figures/rayleigh-spacing.tex for the data source and regeneration.
	\subfile{figures/rayleigh-spacing.tex}
	\caption{Histogram of the unfolded nearest-neighbor spacings
		between consecutive asymptotic directions: each spacing is
		normalized by the empirical mean density of the
		directions, estimated from all realizations, and only
		spacings in the bulk window $0.1 \leq z \leq 0.9$ are kept
		($204{,}139$ spacings pooled from $5{,}743$ Plancherel
		realizations with $n = 10^6$ boxes, rescaled to mean $1$). The histogram is
		shown against the Rayleigh density
		$\tfrac{\pi}{2}\,s\,e^{-\pi s^2/4}$ (solid) and the
		exponential density $e^{-s}$ of independent directions
		(dashed) for comparison. The linear vanishing as
		$s \to 0$, absent from the exponential, is the signature
		of level repulsion.}
	\label{fig:rayleigh-spacing}
\end{figure}

\subsection{The tree of trees and its structure}

The boundaries separating neighboring trees of the jeu de taquin forest,
together with the two coordinate axes, organize the forest into a
hierarchical structure. Crossing a large arc from the $x$-axis to the
$y$-axis, one meets the trees in the order of their asymptotic
directions; moving inward toward the origin, trees terminate at their
cusps and the boundaries flanking a vanishing tree merge. Iterating this
merging, the boundaries appear to form the complete infinite binary
tree, rooted at the origin $(0,0)$: its nodes are the cusps --- one per
tree, by \cref{thm:cusp-in-tree} --- its edges are the tree
boundaries, and its ends correspond to the asymptotic directions, which
are dense in $[0,1]$ by \cref{prop:pn-tree-properties}. We have not made
the notion of tree boundary precise here, and we record this
binary-tree structure as a heuristic rather than a theorem
(\cref{fig:competition-tree}).

\begin{figure}[t]
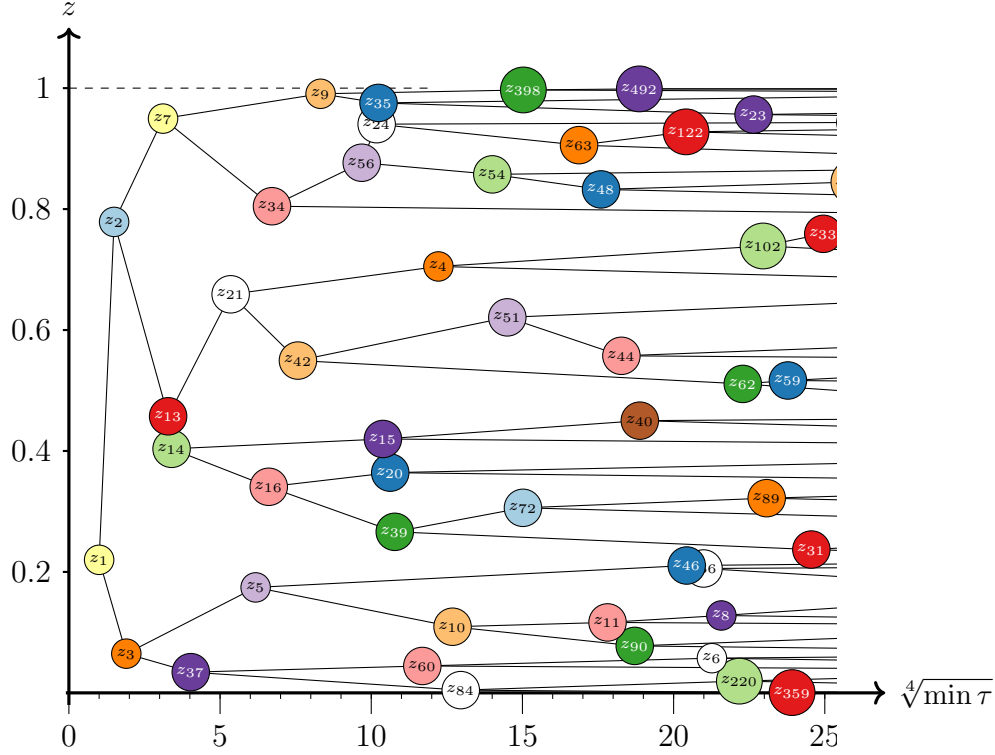

	\centering
	\subfile{figures/binary-tree-competition.tex}
	\caption{The tree of trees of the same Plancherel-random tableau as in
		\cref{fig:jdt-forest}, drawn in
		the coordinates $(\sqrt[4]{\min \tree},\, z)$: the horizontal
		coordinate of a node is the fourth root of its bifurcation
		time $\min \tree$ and its vertical coordinate is the
		asymptotic direction $z$ of the corresponding tree. Each node is a
		cusp --- equivalently, a bifurcation point at which a unique tree
		of the forest originates --- and the edges are the boundaries
		between neighboring trees, exhibiting the binary-tree structure. A
		node is labelled $z_n$, where $n$ is the lifetime of its tree.}
	\label{fig:competition-tree}
\end{figure}

In this description the combinatorial type of the tree is deterministic
--- it is always the infinite binary tree --- while all the randomness
is carried by the embedding into the quarter-plane: each node is placed
at the cusp of its tree, so that the cusp point process is exactly the
image of the node set under this
embedding. The two viewpoints, geometric and combinatorial, thus
describe a single object. The full ``wired'' tree carries more
information than its nodes alone: the scaling limit of the boundaries
themselves, rather than only of the cusps, should be the genealogy of a
system of coalescing Brownian motions, once more in agreement with the
P\'olya web~\cite{BSTU2026polyaWeb}, whose complementary boundary tree
plays precisely this role.

\subsection{An infinite random permutation of $\N$}
\label{subsec:random-permutation}

The trees of the jeu de taquin forest can be enumerated in two natural
ways, and comparing the enumerations produces a random permutation of
$\N$.

On one hand, \cref{prop:lifetime-bijection} assigns to each tree $\tree$
its \emph{lifetime} $n(\tree) \in \N$, and by \cref{thm:main-result} and
\cref{thm:pn-trajectory-tree-bijection} the tree of lifetime $n$ has
asymptotic direction $z_n$; ordering the trees by lifetime is therefore
the same as ordering them by the index of the input value they realize.
On the other hand, the cusps furnish a purely geometric ordering:
list the trees by increasing bifurcation time $\min \tree$. These
values are distinct, so this is a genuine total order. Reading
off the lifetime of the tree placed $k$-th in the cusp order
defines a map
\[
	\sigma \colon \N \to \N,
\]
a bijection because the two orderings enumerate the same set of trees.
The value $\sigma(1) = 1$ is forced: the tree of smallest
bifurcation time is the one at the origin (it contains the box
$(0,0)$ with entry $1$), and the tree at the origin has
lifetime $1$.

For the running example --- the tableau of \cref{fig:jdt-forest},
displayed as a tree of trees in \cref{fig:competition-tree} --- reading
the nodes from left to right, that is, by increasing bifurcation
time $\min \tree$, and
recording the lifetime $n$ of each labelled node $z_n$ yields the first
values
%
% Computed from the seed123 annotated-forest data by
% scripts/compute_sigma_permutation.py in the companion trees-SageMath
% repository; the same data underlies \cref{fig:competition-tree}.
\[
	\sigma = (1, 2, 3, 7, 13, 14, 37, 21, 5, 16, 34, 42, 9, 56, 24,
	\dots).
\]
Already this short prefix is visibly irregular, and the lifetimes that
appear do not form an initial segment of $\N$: a tree of small
lifetime can have a large bifurcation time, having been born early
but coming to rest far from the origin.

The permutation $\sigma$ is a combinatorial fingerprint of the inverse
RSK map, measuring how far the geometric arrangement of the cusps
departs from the order in which the trees are born. In simulations it
appears highly irregular, and we know of no simple description of its
law; its displacement statistics $|\sigma(k) - k|$ and its cycle
structure would be interesting to understand. Unlike the cusp point
process and the tree structure above, $\sigma$ has no counterpart in the
P\'olya web, which possesses cusps and a cusp order but no notion of
input index: it is genuinely a feature of the RSK correspondence.

\subsection{Clustering of boundary-box visits}

\cref{thm:pn-infinite-boundary-boxes} shows that, almost surely, every
jeu de taquin trajectory passes through infinitely many boundary boxes
of the tree containing it. Each boundary box is of one of two types,
according to which diagonal neighbor leaves the tree: call it of type
NW if $(x-1, y+1)$ lies outside the tree, and of type SE if
$(x+1, y-1)$ does. It remains open whether a trajectory must visit
infinitely many boxes of \emph{each} type, or whether it might visit
only finitely many of one type while visiting infinitely many of the
other.

Numerical experiments suggest a more detailed picture. Both types do
seem to occur infinitely often, but the visits are strongly
\emph{clustered}: the trajectory runs for a long stretch alongside one
side of its tree before crossing to the other. Counting the transitions
between the two types along a trajectory truncated to the window
$\{x + y < n\}$, the number of transitions appears to grow very slowly
--- roughly like the square root of the number of boundary-box visits,
and hence like a small power of $n$. Making this growth precise, and
identifying the limiting law of the alternation pattern, is another open
problem. Here too the P\'olya web should provide a tractable setting, in
which the joint behaviour of a tree's two bounding boundaries and the
trajectory running between them could be analyzed, although
\cite{BSTU2026polyaWeb} does not address this question.

\section*{Acknowledgments}

P.~Śniady was supported by the National Science Centre, Poland,
grant~2025/59/B/ST1/01258.

\section*{Data and code availability}

The software that produces the figures and the Monte Carlo statistics of
this paper is openly archived at
\href{https://doi.org/10.5281/zenodo.21361261}{doi:10.5281/zenodo.21361261}
(Apache-2.0); the accompanying dataset of Plancherel-random geodesic trees
is archived at
\href{https://doi.org/10.5281/zenodo.21360601}{doi:10.5281/zenodo.21360601}
(CC-BY-4.0).

\biblio

\end{document}

%% file: figures/FIGURE-iyt.tex
\begin{tikzpicture}[visible node/.style={font=\large}]
	% --- truncated grid + entries: clip a large grid so its NE corner is cut
	%     at (5.2,5.2); the boxes run off the top and right -> extends onward
	\begin{scope}
		\clip (0,0) rectangle (5.2,5.2);
		\draw (0,0) grid (9,9);
		\input{autogenerated/jdt-slide-05x05-bare-content.tex}
	\end{scope}

	% --- ellipses (unclipped): the tableau continues to infinity ---
	\node at (5.6,0.5) {$\cdots$};   % rightward along the rows
	\node at (0.5,5.6) {$\vdots$};   % upward along the columns
	\node at (5.6,5.6) {$\iddots$};  % diagonally to infinity

	% --- coordinate axes (French convention), unclipped: arrows + ticks + nums
	\draw[very thick, ->] (0,0) -- (6.4,0) node[anchor=west] {$x$};
	\draw[very thick, ->] (0,0) -- (0,6.4) node[anchor=south] {$y$};
	\foreach \i in {0,...,5} {
		\draw[thick] (\i,0) -- (\i,-0.2) node[anchor=north] {\small$\i$};
		\draw[thick] (0,\i) -- (-0.2,\i) node[anchor=east] {\small$\i$};
	}
\end{tikzpicture}

%% file: figures/jdt-slide-09x09-before.tex
% Grid style: white ultra thick lines
\renewcommand{\gridcommand}{\draw[white,ultra thick] (0,0) grid (10,10);}

% Load color scheme keyed by underlying input index (z_i)
\input{autogenerated/_jdt-slide-color-scheme.tex}

\begin{tikzpicture}[scale=1,
	competition interface/.style={draw=blue},
	edge/.style={very thick,->},
	% Changed edges and flip boxes are not highlighted in the paper:
	% 'edge changed' renders like a normal edge, 'highlight box' is
	% invisible.
	edge changed/.style={very thick,->},
	jdt path/.style={line width=8pt, blue, opacity=0.35,
		line cap=round, rounded corners=2pt},
	highlight box/.style={draw=none, minimum size=0.6cm, inner sep=0pt},
	boundary/.style={black},
	invisible node/.style={circle, draw=none, fill=none,
		minimum size=0.5cm, inner sep=0pt},
	visible node/.style={font=\small},
	annotation/.style={rotate=45, font=\tiny}
]

	% Restrict the visible/measured region to the meaningful window.
	\clip (-1,-1) rectangle (12,12);

	% Main content: colored regions, values, jdt path, and forest edges
	\begin{scope}
		\subfile{autogenerated/jdt-slide-09x09-before-content.tex}
	\end{scope}

	% Coordinate axes with arrows
	\draw[very thick, ->] (0,0) -- (11,0) node[anchor=west] {$x$};
	\draw[very thick, ->] (0,0) -- (0,11) node[anchor=south] {$y$};

	% X-axis tick marks and labels
	\foreach \x in {0,...,10} {
		\draw[thick] (\x,0) -- (\x,-0.2) node[anchor=north] {\tiny$\x$};
	}

	% Y-axis tick marks and labels
	\foreach \y in {0,...,10} {
		\draw[thick] (0,\y) -- (-0.2,\y) node[anchor=east] {\tiny$\y$};
	}

	% Boundary annotations ($z_i$ circles), one per tree region
	\subfile{autogenerated/jdt-slide-09x09-before-annotations.tex}

\end{tikzpicture}

%% file: figures/FIGURE-jdt_competition_LSVK_ultra2-full.tex
\renewcommand{\gridcommand}{}

% The colour scheme is ~730 \colorlet lines.  Loaded normally inside an
% \hbox (as \resizebox creates) each line-ending newline becomes an
% interword space, inflating the box width by thousands of points and
% shrinking the drawing to a speck.  \endlinechar=-1 suppresses those
% end-of-line spaces; it is set without a group so the (locally scoped)
% \colorlet definitions remain visible to the picture below.
\endlinechar=-1\relax
\input{autogenerated/_common_color_scheme.tex}%
\endlinechar=13\relax

\begin{tikzpicture}[scale=0.18,
competition interface/.style={draw=blue},
boundary/.style={thick, black},
diagram boundary/.style={draw=none}, % LSVK limit-shape curve dropped
jdt path/.style={draw=none}, % trajectory overlay suppressed
truncate/.style={opacity=0}, % truncation removed: full forest shown
annotation/.style={rotate=45, font=\tiny}
]

	\begin{scope}
		\subfile{autogenerated/jdt_tree_interface_LSVK_ultra2.tex}
	\end{scope}

	% Draw main axes with arrows
	\draw[very thick, ->] (0,0) -- (55,0) node[anchor=west] {$x$};
	\draw[very thick, ->] (0,0) -- (0,55) node[anchor=south] {$y$};

	% X-axis tick marks and labels
	\foreach \x in {10,20,...,50} {
		\draw[thick] (\x,0) -- (\x,-1) node[anchor=north] {$\x0$};
	}
	\draw[thick] (0,0) -- (0,-1) node[anchor=north] {$0$};

	% Y-axis tick marks and labels
	\foreach \y in {10,20,...,50} {
		\draw[thick] (0,\y) -- (-1,\y) node[anchor=east] {$\y0$};
	}
	\draw[thick] (0,0) -- (-1,0) node[anchor=east] {$0$};

\end{tikzpicture}

%% file: figures/jdt-slide-09x09-after.tex
% Grid style: white ultra thick lines
\renewcommand{\gridcommand}{\draw[white,ultra thick] (0,0) grid (10,10);}

% Load color scheme keyed by underlying input index (z_i)
\input{autogenerated/_jdt-slide-color-scheme.tex}

\begin{tikzpicture}[scale=1,
	competition interface/.style={draw=blue},
	edge/.style={very thick,->},
	% Changed edges and flip boxes are not highlighted in the paper:
	% 'edge changed' renders like a normal edge, 'highlight box' is
	% invisible.
	edge changed/.style={very thick,->},
	jdt path/.style={line width=8pt, blue, opacity=0.35,
		line cap=round, rounded corners=2pt},
	highlight box/.style={draw=none, minimum size=0.6cm, inner sep=0pt},
	boundary/.style={black},
	invisible node/.style={circle, draw=none, fill=none,
		minimum size=0.5cm, inner sep=0pt},
	visible node/.style={font=\small},
	annotation/.style={rotate=45, font=\tiny}
]

	% Restrict the visible/measured region to the meaningful window.
	\clip (-1,-1) rectangle (12,12);

	% Main content: colored regions, slid values, jdt path, and forest edges
	\begin{scope}
		\subfile{autogenerated/jdt-slide-09x09-after-content.tex}
	\end{scope}

	% Coordinate axes with arrows
	\draw[very thick, ->] (0,0) -- (11,0) node[anchor=west] {$x$};
	\draw[very thick, ->] (0,0) -- (0,11) node[anchor=south] {$y$};

	% X-axis tick marks and labels
	\foreach \x in {0,...,10} {
		\draw[thick] (\x,0) -- (\x,-0.2) node[anchor=north] {\tiny$\x$};
	}

	% Y-axis tick marks and labels
	\foreach \y in {0,...,10} {
		\draw[thick] (0,\y) -- (-0.2,\y) node[anchor=east] {\tiny$\y$};
	}

	% Boundary annotations ($z_i$ circles), one per tree region
	\subfile{autogenerated/jdt-slide-09x09-after-annotations.tex}

\end{tikzpicture}

%% file: figures/FIGURE-curvilinear.tex
	
	% baseline at the origin so the x-axis (y=0) aligns with panel (b)
	\begin{tikzpicture}[scale=3, baseline={(0,0)}]
		\coordinate (origin) at (0,0);
		
		\pgfmathsetmacro{\xpoint}{0.79546}
		\pgfmathsetmacro{\ypoint}{0.49243}

		% Clipped filled region
		\begin{scope}	
			\clip (0,0) -- (\xpoint, \ypoint) -- (2.1,0);	
			\filldraw[fill=Pastel1-C, draw=none, smooth] 
			(origin) -- plot file {vershik-kerov-FR.dat} -- cycle;
		\end{scope}		

		\draw[help lines, dashed] (0,0) grid (2.2,2.2);
		
		% Main curve
		\draw[ultra thick, Set1-B, smooth] plot file {vershik-kerov-FR.dat};
		
		% French coordinate system (xy) - original axes
		\draw[->] (-0.1,0) -- (2.2,0) node[right] {$x$};
		\draw[->] (0,-0.1) -- (0,2.2) node[above] {$y$};
		
		% Russian coordinate system (uv) axes
		% u-axis: direction (1,-1), passing through origin
		% v-axis: direction (1,1), passing through origin
		\draw[->, red, thick] (-0.1,0.1) -- (1.2,-1.2) node[right,red!50!black] {$u$};
		\draw[->, red, thick] (-0.1,-0.1) -- (1.5,1.5) node[above right,red!50!black] {$v$};
		
		% X-axis ticks and labels (French system)
		\foreach \x in {1,2} {
			\draw (\x,0) -- (\x,-0.05) node[below] {$\x$};
		}
		
		% Y-axis ticks and labels (French system)
		\foreach \y in {1,2} {
			\draw (0,\y) -- (-0.05,\y) node[left] {$\y$};
		}
		
		% U-axis ticks and labels (Russian system)
		% For u = k on u-axis (v = 0): x = k/2, y = -k/2
		\foreach \u in {1,2} {
			\pgfmathsetmacro{\xtick}{\u/2}
			\pgfmathsetmacro{\ytick}{-\u/2}
			\draw[red] (\xtick,\ytick) -- (\xtick-0.05,\ytick-0.05) 
			node[below left, red] {$\u$};
		}
		
		% V-axis ticks and labels (Russian system)  
		% For v = k on v-axis (u = 0): x = k/2, y = k/2
		\foreach \v in {1,2} {
			\pgfmathsetmacro{\xtick}{\v/2}
			\pgfmathsetmacro{\ytick}{\v/2}
			\draw[red] (\xtick,\ytick) -- (\xtick-0.05,\ytick+0.05) 
			node[above left, red] {$\v$};
		}
		
		% Optional: Add minor ticks for French system
		\foreach \x in {0.5,1.5} {
			\draw (\x,0) -- (\x,-0.025);
		}
		\foreach \y in {0.5,1.5} {
			\draw (0,\y) -- (-0.025,\y);
		}
		
		% Key point marking
		\draw[dashed] (0,0) -- (\xpoint,\ypoint);
		\fill[Paired-D,draw=black] (\xpoint,\ypoint) circle (0.03);
		\node[above right=2.5pt, fill=white, inner sep=1pt] at (\xpoint,\ypoint) 
		{\textcolor{Paired-D!50!black}{\ $(\operatorname{RSKcos} z, \operatorname{RSKsin} z)$}};
		
		% Parameter label: the green curvilinear triangle has area = z
		% (Lemma thm:curvilinear-triangle), so label the region accordingly.
		\node at (0.92, 0.16) {\small area $= z$};

	\end{tikzpicture}
	

%% file: figures/FIGURE-rsk-polar.tex
\tikzset{
	rhoone/.style  ={Set1-B, very thick, smooth},       % rho = 1 (emphasized)
	rhoint/.style  ={Set1-B!80, semithick, smooth},     % integer rho, solid
	rhohalf/.style ={Set1-B!55, thin, densely dashed, smooth}, % half-int, dashed
	raystyle/.style={Set1-E!85, thin},
	alphalab/.style={Set1-E!60!black},
}

% Ray endpoints: the (RSKcos, RSKsin) values above, times 8 (so each ray
% extends past the clipping box and is then clipped).
\def\raylist{%
	11.236/1.081, 8.746/2.283, 6.774/3.612, 5.093/5.093,%
	3.612/6.774, 2.283/8.746, 1.081/11.236}

% scale=3 and baseline at the origin, matching panel (a)
\begin{tikzpicture}[scale=3, baseline={(0,0)}]

	% ----- curvilinear grid, clipped to the window [0,2.2]x[0,2.2] -----
	\begin{scope}
		\clip (-0.03,-0.03) rectangle (2.23,2.23);

		% alpha = const : straight rays from the origin (equally spaced alpha)
		\foreach \rx/\ry in \raylist {
			\draw[raystyle] (0,0) -- (\rx,\ry);
		}

		% rho = const, half-integer values: dashed
		\foreach \r in {0.5,1.5,2.5,3.5} {
			\begin{scope}[scale=\r]
				\draw[rhohalf] plot file {vershik-kerov-FR.dat};
			\end{scope}
		}

		% rho = const, integer values >= 2: solid
		\foreach \r in {2,3} {
			\begin{scope}[scale=\r]
				\draw[rhoint] plot file {vershik-kerov-FR.dat};
			\end{scope}
		}

		% emphasized unit curve rho = 1 (the LSVK limit shape), solid
		\draw[rhoone] plot file {vershik-kerov-FR.dat};
	\end{scope}

	% ----- French coordinate axes (drawn over the grid) -----
	% IDENTICAL axis style and length to panel (a) / FIGURE-curvilinear.tex
	\draw[->] (-0.1,0) -- (2.2,0) node[right] {$x$};
	\draw[->] (0,-0.1) -- (0,2.2) node[above] {$y$};
	\foreach \x in {1,2} {
		\draw (\x,0) -- (\x,-0.05) node[below] {$\x$};
	}
	\foreach \y in {1,2} {
		\draw (0,\y) -- (-0.05,\y) node[left] {$\y$};
	}
	\foreach \x in {0.5,1.5} {
		\draw (\x,0) -- (\x,-0.025);
	}
	\foreach \y in {0.5,1.5} {
		\draw (0,\y) -- (-0.025,\y);
	}

	% ----- light labels -----
	% the seven equally spaced alpha = k/8 values, just outside the box
	% along their rays (denominators kept as 8 to show equal spacing)
	\node[alphalab, anchor=west]  at (2.25,0.21) {$\tfrac{1}{8}$};
	\node[alphalab, anchor=west]  at (2.25,0.57) {$\tfrac{2}{8}$};
	\node[alphalab, anchor=west]  at (2.25,1.17) {$\tfrac{3}{8}$};
	\node[alphalab, anchor=south west] at (2.21,2.21) {$\alpha=\tfrac{4}{8}$};
	\node[alphalab, anchor=south] at (1.17,2.25) {$\tfrac{5}{8}$};
	\node[alphalab, anchor=south] at (0.57,2.25) {$\tfrac{6}{8}$};
	\node[alphalab, anchor=south] at (0.21,2.25) {$\tfrac{7}{8}$};

	% a couple of rho values along the alpha = 4/8 diagonal ray
	\node[Set1-B!60!black, fill=white, inner sep=1pt] at (0.637,0.637)
		{$\rho=1$};
	\node[Set1-B!60!black, fill=white, inner sep=1pt] at (1.273,1.273)
		{$\rho=2$};

\end{tikzpicture}

%% file: figures/FIGURE-boundary_box-before.tex
\input{_boundary-box-defs.tex}

\begin{tikzpicture}[scale=1.5]
	\begin{scope}
		\clip (-0.4,-1.4) rectangle (3.4,2.4);
		\boundaryboxcontent{a}
		\draw[edge changed] (0.72,0.5) -- (1.28,0.5);
	\end{scope}
	\boundaryboxlabels
\end{tikzpicture}

%% file: figures/FIGURE-boundary_box-after.tex
\input{_boundary-box-defs.tex}

\begin{tikzpicture}[scale=1.5]
	\begin{scope}
		\clip (-0.4,-1.4) rectangle (3.4,2.4);
		\boundaryboxcontent{b}
		\draw[edge changed] (0.5,0.72) -- (0.5,1.28);
	\end{scope}
	\boundaryboxlabels
\end{tikzpicture}

%% file: figures/FIGURE-jdt_competition_LSVK_ultra2-cusps.tex
\renewcommand{\gridcommand}{}

% \endlinechar=-1 suppresses the end-of-line spaces from the ~730
% \colorlet lines (harmless here in vertical mode, but fatal if this
% picture is ever placed in an \hbox, e.g. inside \resizebox); set
% without a group so the locally scoped colours stay visible below.
\endlinechar=-1\relax
\input{autogenerated/_common_color_scheme.tex}%
\endlinechar=13\relax

\centering
\begin{tikzpicture}[scale=0.18,
competition interface/.style={draw=blue},
boundary/.style={thick, black},
diagram boundary/.style={draw=none}, % LSVK limit-shape curve dropped
jdt path/.style={draw=none}, % trajectory overlay suppressed
truncate/.style={opacity=0}, % truncation removed: full forest shown
cusp halo/.style={white, line width=2.8pt, line cap=round},
cusp mark/.style={red, line width=1.2pt, line cap=round},
annotation/.style={rotate=45, font=\tiny}
]

	% Marker: a bold red X on a white casing, so it stays visible on any
	% wedge colour.  \cusphalfarm is the half-length of each cross arm.
	\newcommand{\cusphalfarm}{1.3}
	\newcommand{\cuspX}[2]{%
		\draw[cusp halo] (#1-\cusphalfarm,#2-\cusphalfarm)
			-- (#1+\cusphalfarm,#2+\cusphalfarm)
			(#1-\cusphalfarm,#2+\cusphalfarm)
			-- (#1+\cusphalfarm,#2-\cusphalfarm);
		\draw[cusp mark] (#1-\cusphalfarm,#2-\cusphalfarm)
			-- (#1+\cusphalfarm,#2+\cusphalfarm)
			(#1-\cusphalfarm,#2+\cusphalfarm)
			-- (#1+\cusphalfarm,#2-\cusphalfarm);%
	}

	% Subdued forest: fade the wedges so the cusp markers dominate.
	\begin{scope}[opacity=0.35]
		\subfile{autogenerated/jdt_tree_interface_LSVK_ultra2.tex}
	\end{scope}

	% Cusp markers (bold red X with white casing), on top at full strength.
	\input{autogenerated/cusp-markers-LSVK-ultra2.tex}

	% Draw main axes with arrows
	\draw[very thick, ->] (0,0) -- (55,0) node[anchor=west] {$x$};
	\draw[very thick, ->] (0,0) -- (0,55) node[anchor=south] {$y$};

	% X-axis tick marks and labels
	\foreach \x in {10,20,...,50} {
		\draw[thick] (\x,0) -- (\x,-1) node[anchor=north] {$\x0$};
	}
	\draw[thick] (0,0) -- (0,-1) node[anchor=north] {$0$};

	% Y-axis tick marks and labels
	\foreach \y in {10,20,...,50} {
		\draw[thick] (0,\y) -- (-1,\y) node[anchor=east] {$\y0$};
	}
	\draw[thick] (0,0) -- (-1,0) node[anchor=east] {$0$};

\end{tikzpicture}

%% file: figures/cutoff-sweep.tex
{% group: scope the named styles so the legend swatches can reuse them
\tikzset{
	data/.style={blue!70!black, very thick},
	theory/.style={green!55!black, very thick, densely dashed},
}%
\begin{tikzpicture}[x=9cm, y=5cm]
	% theoretical line K*t^{1/4} and the Monte-Carlo curve (coords in [0,1]^2)
	\input{autogenerated/cutoff-sweep-content.tex}

	% frame
	\draw[thin] (0,0) rectangle (1,1);
	% radial axis  R = t^{1/4} in [0, 31.623]  (x_norm = R/31.623)
	\foreach \R in {0,5,10,15,20,25,30}
		\draw ({\R/31.623},0) -- ({\R/31.623},-0.02) node[below] {$\R$};
	\node[below] at (0.5,-0.12) {$\sqrt[4]{t}$};
	% count axis  cusps per realization in [0,62]  (y_norm = N/62)
	\foreach \N in {0,10,20,30,40,50,60}
		\draw (0,{\N/62}) -- (-0.012,{\N/62}) node[left] {$\N$};
	\node[rotate=90, anchor=south] at (-0.085,0.5) {cusps per realization};

	% legend -- swatches use the SAME styles, so the conjecture line dashes
	\node[anchor=south east, draw, fill=white, font=\small, inner sep=3pt]
		at (0.97,0.04) {
		\begin{tabular}{@{}l@{\ }l@{}}
			\tikz[baseline=-0.5ex]{\draw[data] (0,0) -- (14pt,0);} & Monte Carlo\\
			\tikz[baseline=-0.5ex]{\draw[theory] (0,0) -- (14pt,0);} &
				conjecture $K\,\sqrt[4]{t}$\\
		\end{tabular}};
\end{tikzpicture}%
}%

%% file: figures/cusp-direction-cdf.tex
\begin{tikzpicture}[x=9cm, y=5cm]
	\tikzset{
		resA/.style={red!75!black, very thick},
		resB/.style={orange!85!black, thick},
		resC/.style={green!55!black, thick},
		resD/.style={blue!70!black, thick},
		resE/.style={violet!70!black, thick},
	}
	\draw[thin] (0,0) rectangle (1,1);
	% the empirical CDFs, one per radial shell
	\input{autogenerated/cusp-direction-cdf-content.tex}
	% direction axis z in [0,1]
	\foreach \p in {0,0.2,0.4,0.6,0.8,1}
		\draw (\p,0) -- (\p,-0.02) node[below] {$\p$};
	\node[below] at (0.5,-0.10) {cusp direction $z \in [0,1]$};
	% CDF axis
	\foreach \q in {0,0.2,0.4,0.6,0.8,1}
		\draw (0,\q) -- (-0.012,\q) node[left] {$\q$};
	\node[rotate=90, anchor=south] at (-0.075,0.5)
		{empirical CDF of cusp directions};
	% legend
	\node[anchor=north west, draw, fill=white, font=\small, inner sep=3pt]
		at (1.04,1.0) {
		\begin{tabular}{@{}l@{\ }l@{}}
			\textcolor{red!75!black}{\rule[2pt]{12pt}{1.6pt}} & $\sqrt[4]{t}\in[5,10]$\\
			\textcolor{orange!85!black}{\rule[2pt]{12pt}{1.2pt}} & $\sqrt[4]{t}\in[10,15]$\\
			\textcolor{green!55!black}{\rule[2pt]{12pt}{1.2pt}} & $\sqrt[4]{t}\in[15,20]$\\
			\textcolor{blue!70!black}{\rule[2pt]{12pt}{1.2pt}} & $\sqrt[4]{t}\in[20,25]$\\
			\textcolor{violet!70!black}{\rule[2pt]{12pt}{1.2pt}} & $\sqrt[4]{t}\in[25,30]$\\
		\end{tabular}};
\end{tikzpicture}

%% file: figures/rayleigh-spacing.tex
\begin{tikzpicture}[x=0.17\textwidth, y=5.8cm,
	histbar/.style={fill=blue!18, draw=blue!35, line width=0.2pt},
	raycurve/.style={red, very thick},
	expcurve/.style={gray, thick, densely dashed}]

	% bars (data) then the two theory curves (auto-generated coordinates)
	\input{autogenerated/rayleigh-spacing-content.tex}

	% axes
	\draw[thin] (0,0) -- (4.1,0);
	\draw[thin] (0,0) -- (0,0.88);

	% x ticks (rescaled spacing s)
	\foreach \xx in {0,1,2,3,4}
		\draw (\xx,0) -- (\xx,-0.018) node[below] {\small $\xx$};

	% y ticks (density)
	\foreach \yy in {0,0.2,0.4,0.6,0.8}
		\draw (0,\yy) -- (-0.05,\yy) node[left] {\small $\yy$};

	% axis labels
	\node[below] at (2,-0.085) {rescaled spacing $s$};
	\node[rotate=90, anchor=south] at (-0.42,0.4) {\small density $P(s)$};

	% legend (upper right)
	\draw[raycurve] (2.45,0.80) -- (2.95,0.80);
	\node[right, font=\small] at (3.0,0.80)
		{Rayleigh $\tfrac{\pi}{2}s\,e^{-\pi s^2/4}$};
	\draw[expcurve] (2.45,0.71) -- (2.95,0.71);
	\node[right, font=\small] at (3.0,0.71) {exponential $e^{-s}$};
	\fill[histbar] (2.55,0.605) rectangle (2.85,0.655);
	\node[right, font=\small] at (3.0,0.63) {data};

\end{tikzpicture}

%% file: figures/binary-tree-competition.tex
% Clipping command to show specific horizontal range
\newcommand{\clipcommand}{\clip(-0.4, -0.1) rectangle (25.4,1.15);}

% Load color scheme for tree regions (dedicated: includes text-* colors)
\input{autogenerated/binary-tree-color-scheme.tex}

\begin{tikzpicture}[yscale=8, xscale=0.4,
	node_style/.style={circle, draw=black, fill=white, inner sep=1pt, outer sep=0pt},
	my edge style/.style={black},
	axis label/.style={fill=white, font=\small}
]

	% Y-axis tick marks and labels (z values)
	\foreach \y in {0.2,0.4,...,1.0} {
		\draw[thick] (0,\y) -- (-0.2,\y) node[left] {$\pgfmathprintnumber{\y}$};
	}

	% Draw axes through origin (0,0)
	\draw[very thick, ->] (0,-0.02) -- (0,1.1) node[above, axis label]{$z$};
	\draw[very thick, ->] (-0.1,0) -- (27,0) node[right, axis label]{$\sqrt[4]{\min\tau}$};

	% Apply clipping
	\clipcommand

	% Horizontal dashed line at z=1
	\draw[dashed] (0,1) -- (12,1);

	% X-axis minor tick marks (every unit)
	\foreach \x in {0,...,27} {
		\draw (\x,0) -- (\x,-0.01);
	}

	% X-axis major tick marks and labels (every 5 units)
	\foreach \x in {0,5,...,25} {
		\draw (\x,0) -- (\x,-0.035) node[below] {$\x$};
	}

	% Main content: autogenerated binary tree with competition interfaces
	\tiny
	\input{autogenerated/binary-tree-competition-content.tex}

\end{tikzpicture}